\documentclass[11pt]{amsart}

\usepackage{sigma,verbatim,hyperref}

\usepackage{setspace}

\begin{document}

\title{Popular products and continued fractions}

\author{Nikolay Moshchevitin, Brendan Murphy, and Ilya Shkredov}
\address[Nikolay Moshchevitin]{Lomonosov Moscow State University, Division of Mathematics, Moscow, Russia}
\email{moshchevitin@gmail.com}
\address[Brendan Murphy]{School of Mathematics, University of Bristol, Bristol, UK, and
Heilbronn Institute of Mathematical Research}
\email{brendan.murphy@bristol.ac.uk}
\address[Ilya Shkredov]{Steklov Mathematical Institute,ul. Gubkina, 8, Moscow, Russia, 119991,
IITP RAS, Bolshoy Karetny per. 19, Moscow, Russia, 127994, and
MIPT, Institutskii per. 9, Dolgoprudnii, Russia, 141701}
\email{ilya.shkredov@gmail.com}

\date{\today}

\maketitle

\begin{abstract}
We prove bounds for the popularity of products of sets with weak additive structure, and use these bounds to prove results about continued fractions.
Namely, we obtain a nearly sharp upper bound for the cardinality of Zaremba's set modulo $p$.
\end{abstract}

\listoffixmes

\section{Introduction}

This paper is about a variation of the \emph{sum-product problem}, and the application of such results to problems on continued fractions.

\subsection{The sum-product problem}
The sum-product problem is to show quantitatively that a finite subset of a ring cannot be approximately closed under addition and multiplication, unless it is approximately a subring.
Originally, Erd\H{o}s and Szemer\'edi \cite{erdos1983products} considered a finite set $A$ of integers and asked if $A$ must grow under either addition or multiplication.
More precisely, they considered the \emph{sum set $A+A=\{a+a'\colon a,a'\in A\}$}\/ and \emph{product set $AA=\{aa'\colon a,a'\in A\}$}\/ and asked if we must have
\[
\max(|A+A|,|AA|)\gg |A|^{1+\delta}
\]
for some $\delta>0$.

We study a related phenomenon: if $A$ is a subset of $\F_p$ and $A+B$ is small for some set $B$, which may be much smaller than $A$, then for any non-zero element $x\in AA$, the number of ways to write $x=aa'$ with $a,a'\in A$ is $o(|A|)$.
That is, if $A$ is almost invariant under addition by a smaller set, then $AA$ contains no popular products.

\subsection{Summary of results}

Our first type of result shows that if the sumset of $A$ and $B$
\[
A+B=\{a+b\colon a\in A,b\in B\}
\]
is small, then the sumset of the set of reciprocals $A^{-1}$ with any other set $C$ 
\[
A^{-1}+C=\{a^{-1}+c\colon a\in A,c\in C\}
\]
must be large.
These results work when $B$ and $C$ are much smaller than $A$.
See Theorem~\ref{thm:1} and Theorem~\ref{thm:2}.

We use these results to show that if $A+B$ is small, where $B$ may be much smaller than $A$, then $A$ does not have any popular products.
That is, for all $\rho\not=0$,
\[
|A\cap \rho A^{-1}| = |\{(a,a')\in A\times A\colon aa'=\rho\}| = o(|A|).
\]
See Corollaries \ref{cor:1}, \ref{cor:7}, and \ref{cor:2}.

We use these bounds on popular products to bound the number of integers $1\leq a\leq p-1$ such that the continued fraction expansion of $a/p$ has partial quotients bounded by a fixed number $M$.
See Theorem~\ref{thm:7}.

\subsection{Methods}

To prove lower bounds for $\max(|A+B|,|A^{-1}+C|)$, we consider a set $S$ of linear fractional transformations that map at least $|A|$ element of $A^{-1}+C$ to $A+B$.
If both $A+B$ and $A^{-1}+C$ are not much larger than $A$, then $S$ is a set of \emph{rich linear fractional transformations}\/ of $Y=(A+B)\cup (A^{-1}+C)$.
This is related to Elekes' geometric proof \cite{elekes1997number,elekes1997linear} of a lower bound for $\max(|A+B|,|AC|)$; since we need $B$ and $C$ to be much smaller than $A$, our methods of proof are closer to that of the asymmetric sum-product theorem \cite{bourgain2005sum-product,shkredov2017remarks,murphy2017upper,murphy2017group}.

We use the $\ell^2$-flattening method of \cite{bourgain2008uniform-a} to prove asymptotic estimates for the number of rich linear fractional transformations.
See \cite{shkredov2018asymptotic} for similar results and methods.
In addition, a related result was proved by Bourgain~\cite{bourgain2012modular}, framed as an incidence bound for Cartesian product point sets and hyperbolas (corresponding to graphs of linear fractional transformations.)

\subsection{Notation}

Given two sets of finite subsets $A$ and $B$ of a commutative ring, we use $A\pm B$ to denote the \emph{sum set}\/ and \emph{difference set}\/ of $A$ and $B$
\[
A\pm B:=\{a\pm b\colon a\in A,b\in B\}
\]
and $AB$ to denote the \emph{product set}\/ of $A$ and $B$
\[
AB:=\{ab\colon a\in A,b\in B\}
\]
If the elements of $A$ are invertible, we use $A^{-1}$ to denote the set of inverses of elements of $A$.
The \emph{ratio set}\/ of $A$ and $B$ is $A/B=A(B\setminus\{0\})^{-1}$.
If $\rho\not=0$, we use $\rho A$ to denote the set of dilates of elements of $a$ by $\rho$
\[
\rho A:=\{\rho a\colon a\in A\}.
\]

All logarithms are base 2.

We use the standard Vinogradov symbols $\gg$ and $\ll$:
\[
f \ll g\iff \exists C>0 \quad f\leq C g,
\]
and $f\gg g$ if and only if $g\ll f$.
We write $f\asymp g$ if $f\ll g$ and $g\ll f$.
A subscript in the asymptotic notation, such as $f\ll_M g$, means that the implicit constant $C$ depends on the variable $M$.
We have used \emph{little-o notation}\/ in the introduction for brevity; we give precise statements below.

For a real number $x$, we use $\floor x$ to denote the greatest integer less than or equal to $x$, and we use $\ceil x$ to denote the least integer greater than or equal to $x$.
Thus, $\floor x \leq x < \floor x +1$ and $\ceil x -1<x\leq \ceil x$.

We use vertical bars to denote the cardinality of a set, for instance $|A|$.

If $G$ is a group acting on a set $X$, $f\colon G\to\bbC$ has finite support, and $\phi\colon X\to \bbC$, then we define the \emph{convolution}\/ $f\ast \phi\colon X\to\bbC$ by
\[
(f\ast \phi)(x):=\sum_{g\in G}f(g)\phi(g^{-1}x).
\]
A special case of this is when $X=G$ and $G$ acts on itself by left-translation.

\subsection{Organization}

This paper is organized as follows.
\begin{itemize}
\item In Section~\ref{sec:bounds-sums-recipr}, we state lower bounds for $\max(|A+B|,|A^{-1}+C|)$ and use these bounds to derive popular product bounds for sets that are almost invariant under addition with a smaller set.
\item In Section~\ref{sec:appl-cont-fract}, we apply the results from the previous section to a problem in continued fractions.
\item In Section~\ref{sec:bounds-rich-linear}, we prove bounds for the number of ``rich'' linear fractional transformations; this is the tool we use to prove bounds in Section~\ref{sec:bounds-sums-recipr}.
\item In Section~\ref{sec:proofs-bounds-sums} we prove the bounds for sums of reciprocals stated in Section~\ref{sec:bounds-sums-recipr}, using the results in Section~\ref{sec:bounds-rich-linear}.
\item In Section~\ref{sec:ell2-flatt-energ} we prove a $\ell^2$-flattening result for linear fractional transformations acting on the projective line.
\item In Sections~\ref{sec:proof-theor-thm:15} and \ref{sec:proof-theor-thm:5} we prove the results used to prove the rich linear fractional transformations results in Section~\ref{sec:bounds-rich-linear}.
\end{itemize}

\section{Bounds for sums of reciprocals and popular products}
\label{sec:bounds-sums-recipr}

In this section we state two lower bounds (Theorems~\ref{thm:1} and \ref{thm:2}) for sums of a set and its reciprocals, and then derive bounds for popular products.
The proofs of Theorems~\ref{thm:1} and \ref{thm:2} are in Section~\ref{sec:proofs-bounds-sums}, since they require technical results stated in Section~\ref{sec:bounds-rich-linear}.

\begin{thm}
 \label{thm:1}
Let $A, B,$ and $C$ be subsets of $\F_p$, and let $\rho$ be a non-zero element of $\F_p$.

There is a constant $b_0>1$ such that
for all $\epsilon>0$ and all $\delta\leq \frac 14 b_0^{-1/\epsilon}$,
if 
$\min(|B|,|C|)=p^\epsilon$, then for all sufficiently large $p$ we have
\[
|A+B| + |\rho A^{-1}+C| \gg \min(\sqrt{p|A|}, |A|p^\delta).
\]
\end{thm}
In fact, if we write $W=(A+B)\cup (\rho A^{-1}+C)$, then we have
\begin{equation}
  \label{eq:32}
  |A|\leq \frac {|W|^2}p+C_*|W|p^{-\delta(k)},
\end{equation}
where $C_*\geq 6$ is an absolute constant and $\delta(k)=2^{-(k+2)}$, where \[k\gg \log_{\min(|B|,|C|)}p.\] 

Similar results were proved in \cite{shkredov2018asymptotic}, and other results about sums of reciprocals were proved in  \cite{yazici2015growth} and \cite{mojarrad2017conditional}.

Theorem~\ref{thm:1} implies a bound for popular products.
\begin{cor}
  \label{cor:1}
There is a constant $b_0>1$ such that the following holds for all $0<\kappa<1$, $\epsilon>0$, and $\delta\leq\frac 14 b_0^{-1/\epsilon}$. 

Suppose that $A\subseteq X+B$, where $A,X,B\subseteq \F_p$,
$|B+B|\leq \sigma|B|$, $|B|\geq p^\epsilon$, and $|X|\ll \frac{|A|}{|B|}|B|^\kappa$.

If $\rho\not=0$, then
\[
|A\cap \rho A^{-1}|\ll \max \left( \frac{\sigma^2|A|^2|B|^{2\kappa}}p, \frac{\sigma|A||B|^\kappa}{p^{\delta}} \right).
\]
\end{cor}
\begin{proof}
  Put $A_*=A\cap \rho A^{-1}$.
Then
\[
|A_*+B|\leq |A+B|\leq |X+B+B|\leq |X||B+B|\leq \sigma|X||B| \ll \sigma |A||B|^\kappa.
\]
Since $A_*=\rho A_*^{-1}$, we have
\[
|A_*+B| + |\rho A_*^{-1}+B|\ll \sigma|A||B|^{\kappa}.
\]
By Theorem~\ref{thm:1}, 
\[
|A_*+B| + |\rho A_*^{-1}+B|\gg \min( \sqrt{p|A_*|},|A_*|p^{\delta}),
\]
where $\delta\leq b_0^{-1/\epsilon}$.
Combining the last two equations, we have the desired upper bound for $|A_*|$.
\end{proof}

\begin{cor}
  \label{cor:7}
There is a constant $b_0>1$ such that the following holds for all $\epsilon>0$ and $\delta\leq\frac 14 b_0^{-1/\epsilon}$. 

Suppose that $A,B\subseteq\F_p$, $|A+B|\leq\sigma|A|$, and $|B|\geq p^\epsilon$.

If $\rho\not=0$, then
\[
|A\cap \rho A^{-1}|\ll \max \left( \frac{\sigma^2|A|^2}p, \frac{\sigma|A|}{p^{\delta}} \right).
\]
\end{cor}

\begin{thm}
  \label{thm:2}
Fix $0<\tau< 1/8$.
Let $A, B,$ and $C$ be subsets of $\F_p$ such that $B=\{1,\ldots,M\}, C=\{1,\ldots, N\},$ and $1\leq |A|\leq p^{1-\delta}$, where $\delta=0.25\, b_0^{-1/\tau}$ for an absolute constant $b_0>1$.

If $p\gg 1$ and $11\leq\min(|B|,|C|)\leq p^{\tau}$
then
\[
|A+B|+|A^{-1}+C|\geq \frac{|A|}2 \left( \frac{\min(|B|,|C|)}{2} \right)^{\delta/\tau}.
\]
\end{thm}
Theorem~\ref{thm:2} is proved in Section~\ref{sec:proofs-bounds-sums} using Theorem~\ref{thm:5}, stated below.
Our motivation for proving Theorem~\ref{thm:2} is the following corollary.
\begin{cor}
  \label{cor:2}
Suppose that $A\subseteq X+B$, where $A,X,B\subseteq \F_p$, $1\leq |A|\leq p^{1-\kappa}$, $B=\{1,\ldots,|B|\}$, and $|X|\leq \frac{|A|}{|B|}|B|^\kappa$.

If $p\gg 1$ and $11\leq |B|\leq p^{\tau}$, then for $0<\tau<1/8$ and $\kappa=0.25\, b_0^{-1/\tau}$ we have
\[
|A\cap A^{-1}|\leq 2^{4+\kappa/\tau} |A| |B|^{\kappa(1-1/\tau)}.
\]
\end{cor}
\begin{proof}
Let $A_*=A\cap A^{-1}$.
We have $|A_*+B|\leq 2|A||B|^\kappa$.

Since $A_*=A^{-1}_*$, by Theorem~\ref{thm:2} with $\delta=\kappa$ and $0<\tau<1/8$ we have
\[
\frac{|A_*|}4 \left( \frac{|B|}{2} \right)^{\kappa/\tau}
\leq |A_*+B|+|A_*^{-1}+B| \leq 4|A||B|^\kappa.
\]
Hence
\[
|A_*|\leq 2^{4+\kappa/\tau} |A| |B|^{\kappa(1-1/\tau)}.
\]
\end{proof}



\section{Application to continued fractions with bounded partial quotients}
\label{sec:appl-cont-fract}

Here we discuss some problems of representing rational numbers by finite continued fractions.
By the Euclidean algorithm, a rational $a/q\in [0,1], (a,q)=1$ can be uniquely represented as a regular continued fraction
\begin{equation}\label{exe}
\frac{a}{q}=[0;b_1,\dots,b_s]=
\frac{1}{\displaystyle{b_1+\frac{1}{\displaystyle{b_2+
\frac{1}{\displaystyle{b_3 +\dots +
\displaystyle{\frac{1}{b_{s}}} }}}}}} ,\,\,\,\,\, b_s \ge 2.
\end{equation}

Assuming $q$ is known, we use $b_j(a), j=1,\ldots,s=s(a)$, to denote the partial quotients of $a/q$; that is,
\[
\frac aq := [ 0; b_1(a),\ldots,b_{s}(a)].
\]

\subsection{Zaremba's conjecture}

\newcommand{\zk}{\ensuremath{\mathfrak{k}}}

Zaremba's famous conjecture \cite{zaremba1972methode} posits that there is an absolute constant $\zk$ with the following property:
for any positive integer $q$ there exists $a$ coprime to  $q$ such that in the continued fraction expansion (\ref{exe}) all partial quotients are bounded:
\[
 b_j (a) \le \zk,\,\, 1\le j  \le s = s(a).
\]
In fact, Zaremba conjectured that $ \zk=5$.
For large prime $q$, even $ \zk=2$ should be enough, as conjectured by Hensley.

Korobov \cite{korobov1963number} showed that for prime $q$  there exists  $a$,
$(a,q)=1$, such that
\[
\max_\nu b_\nu(a)\ll\log q.
\]
Such a result is also true for composite $q$.
Moreover, Rukavishnikova \cite{rukavishnikova2006probabilistic} proved 
that Korobov's bound holds with positive probability:
\[
\frac{1}{\varphi (q)}\left|\left\{ a\in \mathbb{Z}:\,\,
1\le a \le q,\,\, (a,q) = 1,\,\, \max_{1\le j\le s(a)} b_j (a) \ge T\right\}\right|
\ll \frac{\log q}{T}.
\]
The main results of Rukavishnikova's papers \cite{rukavishnikova2006probabilistic,rukavishnikova2011large} deal with the typical values of the sum of partial quotients  of fractions with a given denominator: she proves an analog of the law of large numbers. 

It is clear that Zaremba's conjecture is true when $q= F_n$ is the $n$-th  Fibonacci number.
Niederreiter \cite{niederreiter1986dyadic} proved that  Zaremba's conjecture is true for $q = 2^\alpha,3^\alpha, \,\, \alpha\in\mathbb{Z}_+$ with $\zk=3$, and for $q=5^\alpha$ with $\zk=4$.                                             
By means a quite similar argument Yodphotong and Laohakosol showed \cite{yodphotong2002proofs} that Zaremba's conjecture is true for $q=6^\alpha$ and $\zk=5$.
Komatsu \cite{komatsu2005zarembas} proved that Zaremba's conjecture is true for $q=7^{r2^r}, r =1,3,5,7,9,11$ and $\zk=3$.
Kan and Krotkova \cite{kan2011quantitative} obtained 
lower bounds for the  number
\[
f=
|
\{ a\pmod{p^m}:\,\, a/p^m =[0;b_1,\dots. b_s],\,\, b_j \le p^n\}|
\]
of fractions with bounded partial quotients and the denominator of the form $p^n$. 
In particular they proved a bound of the form
\[ 
f \ge C(n)  m^\lambda,\,\,\, C(n), \lambda >0.
\]

Recently Bourgain and  Kontorovich \cite{bourgain2011zarembas,bourgain2014zarembas} made significant progress on Zaremba's conjecture.
Consider the set
\[
\mathfrak{Z}_k (N) :=\{
q\le N:\,\, \exists a \,\,\,\text{such that}\,\,\, (a,q) =1,\,\, \,\, a/q=[0;b_1,...,b_s],\,\,\, b_j \le k\}
\]
(so Zaremba's conjecture means that $\mathfrak{Z}_k (N)= \{1,2,...,N\}$).
In a wonderful paper \cite{bourgain2014zarembas} Bourgain and Kontorovich proved that for  $k$ large enough there exists positive $c=c(k)$ such that for $N$ large enough one has
\[
|\mathfrak{Z}_k(N  )|
= N- O(N^{1-c/\log\log N}).
\]
 
For example, it follows from this result that for $k$ large enough the set $\bigcup_N \mathfrak{Z}_k (N)$ contains infinitely many prime numbers.

Another result from \cite{bourgain2014zarembas} states that for $k=50$ the set
\begin{equation}\label{poop}
\bigcup_N \mathfrak{Z}_{50}(N)
\end{equation}
has positive density in $\mathbb{Z}_+$, that is
\[
|\mathfrak{Z}_{50} (N)| \gg N.
\]
 
This result was improved by Frolenkov and Kan \cite{kan2014strengthening,frolenkov2014strengthening,kan2015strengthening,kan2016strengthening,kan2017strengthening}, Huang \cite{huang2015improvement}, and Magee, Oh, and Winter \cite{magee2016uniform}.
In particular, in \cite{kan2016strengthening} Kan proved that the set (\ref{poop}) has positive density 
in $\mathbb{Z}_+$ for $ k = 4$.
 
\subsection{Real numbers with bounded partial quotients}

By $F_M(Q)$  we denote the set of all rational numbers  $\frac{u}{v}, (u,v) = 1$ from $[0,1]$ with all partial quotients in (\ref{exe}) not exceeding $M$ and with $ v\le Q$:
\[
F_M(Q)=\left\{
\frac uv=[0;b_1,\ldots,b_s]\colon (u,v)=1, 0\leq u\leq v\leq Q, b_1,\ldots,b_s\leq M
\right\}.
\]
 By $F_M$ we denote the set of all irrational real numbers  from $[0,1]$ with partial quotients less than or equal to $M$.
From \cite{hensley1992continued} we know that the Hausdorff dimension $w_M$ of the set  $F_M$ satisfies
\begin{equation}
 w_M = 1- \frac{6}{\pi^2}\frac{1}{M} -
\frac{72}{\pi^4}\frac{\log M}{M^2} + O\left(\frac{1}{M^2}\right),
\,\,\, M \to \infty ,
 \label{HHD}
\end{equation}
however here we need simpler result from \cite{hensley1989distribution}, which states that
\begin{equation}\label{oop}
1-w_{M} \asymp  \frac{1}{M}
\end{equation}
with absolute constants in the sign $\asymp$.
Explicit estimates for $\dim F_M$  for certain values of $M$ can be found in \cite{jenkinson2004density}.
In the papers \cite{hensley1989distribution,hensley1990distribution} Hensley gives the bound
\begin{equation}
|F_M(Q)| \asymp_M Q^{2w_M}.\label{QLOW}
\end{equation}
For a fixed $N$ we consider the set 
\[
Z_M(N):= \left\{ a\in \{1,2,...,N-1\}\colon (a,N) = 1, \max_{1\leq j\leq s}b_j(a)\leq M\right\}
\]
of all positive integers $a$ less than $N$ so that the partial quotients of $a/N$ are all bounded by $M$.
For instance, Zaremba's conjecture is that for $M=5$ and all $N$, we have $|Z_M(N)|>0$.

In \cite{moshchevitin2007setsoftheform}, the first author used Hensley's bounds to show that \fxnote{is $p$ prime? In Nikolay's paper, $q$ is not necessarily prime (as far as I can tell) so maybe we should use $N$}
\begin{equation}
\label{ooopo}
|Z_M(p)|\ll_M p^{w_M}.
\end{equation}
Certain upper bounds for $|Z_M(p)|$ were obtained recently in  \cite{david2017equidistribution} by means of Dynamical Systems.
In the next subsection we improve on (\ref{ooopo}) in the case when $N=p$ is a prime number, and give an upper bound that is close to optimal.

\subsection{New results} 

For a prime $p$, we consider the set 
\[
 Z_M(p) =
\left\{a\in\{1,\ldots,p-1\}\colon \max_{1\leq j\leq s(a)}b_j(a) \leq M\right\}
\]
Our main new result is the following theorem.
\begin{thm}
\label{thm:7}
  Given positive $\varepsilon$ there exists $M_0 = M_0(\varepsilon)$ such that for all $ M\ge M_0$ one has
\[
|Z_M(p)|\ll_M p^{2w_M -1 +\varepsilon (1-w_M)}.
\]
\end{thm}
For large values of $M$, the exponent here is close to the optimal exponent $2w_M -1$ that was conjectured in lecture \cite{moshchevitin2017problems}.
One can see that Theorem~\ref{thm:7} improves the bound (\ref{ooopo}) from \cite{moshchevitin2007setsoftheform}.
Some related problems are discussed in the preprint \cite{moshchevitin2012problems}.

Before proving Theorem~\ref{thm:7}, we introduce some auxiliary sets.

Recall that if 
\[
\frac aq = [0; b_1,\ldots,b_n],
\]
then the \emph{$k$th convergent}\/ to $a/q$ is $[0;b_1,\ldots,b_k]$.
We use $u_k$ and $v_k$ to denote coprime integers such that
\[
\frac{u_k}{v_k}=[0;b_1,\ldots,b_k].
\]
When $q$ is understood, we will write $u_k(a)$ and $v_k(a)$ for the convergents $u_k(a)/v_k(a)$ to $a/q$.

The integers $u_k$ and $v_k$ satisfy the following recursion relations: $u_0=0,u_1=1$, and for $k\geq 1$
\[
u_{k+1}=b_{k+1}u_k+u_{k-1},
\]
and $v_0=1,v_1=b_1$, and for $k\geq 1$
\[
v_{k+1}=b_{k+1}v_k+v_{k-1}.
\]
In addition, we have the following error bound for approximating $a/q$ by its convergents:
\begin{equation}
  \label{eq:68}
  \left|\frac aq -\frac{u_k}{v_k} \right| < \frac 1{v_kv_{k+1}}< \frac 1{v_k^2}.
\end{equation}
See \cite[Chapter X]{hardy1954introduction} or \cite[Chapter 1]{vinogradov1954elements} for further properties of continued fractions and convergents.

Let
\[
A=A_M(p)=\{a\in\{1,\ldots,p-1\}\colon \mbox{$b_k(a)\leq M$ for all $k$ such that $v_k(a)\leq\sqrt p$}\}.
\]
That is, $A$ is the set of $a$ such that the partial quotients of all convergent fractions $u/v$ to $a/p$ with $v\leq \sqrt p$ are at most $M$.
Note that $Z_M(p)\subseteq A_M(p)$, and that every convergent $u/v$ to $a/p$, with $a\in A_M(p)$ and $v\leq\sqrt p$ is contained in $Z_M(\sqrt p)$.
Further, the set $A$ has an involution defined by $a\mapsto a^*$, where $aa^*\equiv 1\pmod p$, so when we consider $A$ as a subset of $\F_p$, we have $A=A^{-1}$.

More precisely, if
\[
 \frac{a}{p}=[0;b_1,b_2\dots,b_s],
\]
 with $b_s \ge 2$ then for the inverse element  $a^{*}$ modulo $p$
 defined by $ aa^{*} \equiv 1\pmod{p}$ we have \cite{rukavishnikova2006probabilistic,rukavishnikova2011large}:
\[
 \frac{a^{*}}{p}=[0;b_s,b_{s-1}\dots,b_1] \qquad \mbox{if $s$ is even}
\]
\[
  \frac{a^{*}}{p}=[0; 1,b_s-1,b_{s-1}\dots,b_1] \qquad\mbox{if $s$ is odd}.
\]


\newcommand{\abeta}{\ensuremath{\mathcal{A}_\beta}}

Now we take $\beta$ from the range
\[
  0<\beta \leq\frac{1}{2}
\]
and consider the set
\[
\abeta=\left\{
a\colon \mbox{$\exists \frac uv\in F_M(p^\beta)$ such that $a=\floor{p\frac uv}$}
\right\}.
\]
(Recall that $\frac uv\in F_M(p^\beta)$ if $0\leq u\leq v\leq p^\beta, (u,v)=1$, and all partial quotients of $\frac uv$ are less than $M$.)

\begin{lem}
  \label{lem:5}
For $0<\beta\leq 1/2$, the map $\frac uv\mapsto \floor{p\frac uv}$ from $F_M(p^\beta)$ to $\abeta$ is bijective.
Hence $|\abeta|=|F_M(p^\beta)|\approx_M p^{2\beta\omega_M}$.
\end{lem}
\begin{proof}
By definition, the map $\frac uv\mapsto\floor{p\frac uv}$ from $F_M(p^\beta)$ to $\abeta$ is surjective.
For $0<\beta\leq 1/2$, this map is also injective, since for distinct $\frac{u}{v},\frac{u'}{v'}  \in F_{M}(p^\beta)$ we have
\[
 \left|
 \frac{u}{v}-\frac{u'}{v'} 
\right|\ge \frac{1}{vv'} \geq \frac{1}{p},
\]
hence different $ \frac{u}{v}$ give different $a$.

It follows immediately that $|\abeta | = |F_{M}(p^\beta)|$.
By \eqref{QLOW}, $|F_M(p^\beta)|\asymp_{M} p^{2\beta w_M}$.
\end{proof}

\newcommand{\bbeta}{\ensuremath{\mathcal{B}_\beta}}

Now we define the set of consecutive integers
\[
\bbeta = \left\{ 0, \pm 1,\pm 2 ,\dots, \pm \floor{(M+1)^2p^{1-2\beta}+1}\right\}.
\]

\begin{lem}
  \label{lem:9}
For $A,\abeta,$ and $\bbeta$ defined as above, we have $A\subseteq \abeta+\bbeta$.
\end{lem}
\begin{proof}
The denominators of convergents  $\frac{u_\nu}{v_\nu}$ satisfy the relation
\[
 v_\nu < v_{\nu+1} = b_{\nu+1} v_\nu +v_{\nu-1} \le (b_{\nu+1}+1) v_\nu.
\]
So for any rational $\frac{a}{p}$ with partial quotients $\le M$ and for any $\lambda$  from the interval $ M+1\le \lambda \le p$ there exists a convergent fraction  $\frac{u}{v}$ to $\frac{a}{p}$ with
$ \frac{\lambda}{M+1} \le  v \le \lambda $. We see that every   rational $\frac{a}{p}$ with $ a\in A$ must have a convergent fraction $\frac{u}{v}$ from $F_{M}(p^\beta)$  with $ v \ge \frac{p^\beta}{M+1}$ and for this convergent fraction one has
\[
\left|\frac{a}{p} - \frac{u}{v}\right| \le\frac{1}{v^2} \le \frac{(M+1)^2}{p^{2\beta}}.
\]
This observation implies
\[
\left| a - \floor{p\frac uv}\right| \leq (M+1)^2p^{1-2\beta}+1,
\]
which leads to the desired inclusion $A \subseteq \abeta +\bbeta$.
\end{proof}

\begin{proof}[Proof of Theorem~\ref{thm:7}]
Recall that $|\abeta|\asymp_M p^{2\beta \omega_M}$ and
\[
|\bbeta|=2\floor{(M+1)^2p^{1-2\beta}+1}+1.
\]

Since $A\subseteq \abeta+\bbeta$ and
\[
\bbeta+\bbeta\subseteq \left\{0,\pm \floor{(M+1)^2p^{1-2\beta}+1}\right\}+\bbeta,
\]
we have
\[
|A+\bbeta|\leq |\abeta +\bbeta+\bbeta| \leq 3|\abeta+\bbeta|\leq 3|\abeta||\bbeta|\ll_M p^{1-2\beta(1-\omega_M)}.
\]

Since $A=A^{-1}$, we have 
\[
|A+\bbeta|+|A^{-1}+\bbeta|\ll_M p^{1-2\beta (1-\omega_M)}.
\]
By Theorem~\ref{thm:2} with $\tau=1-2\beta +2\log_p(M+1)$ and $\delta=1-\omega_M$, we have
\begin{equation}
  \label{eq:46}
p^{(1-2\beta)(1-\omega_M)/\tau}|A|\ll_M  p^{1-2\beta (1-\omega_M)}
\end{equation}
provided that
\[
\delta\leq\frac 14 b_0^{-1/\tau}.
\]

Thus
\[
|A|\ll_M p^{\omega_M+(1-2\beta)(1-\omega_M)(1-\tau^{-1})}.
\]

Now we choose $\beta = \frac{1-\varepsilon}{2}$, so that
\[
|A|\ll_M p^{\omega_M+\epsilon(1-\omega_M)(1-1/(\epsilon+2\log_p(M+1)))}\ll_M p^{2\omega_M-1+\epsilon(1-\omega_M)},
\]
provided that
\[
1-\omega_M \leq \frac 14 b_0^{-1/(\epsilon+2\log_p(M+1))}.
\]
For $p$ sufficiently large, it suffices to take $\epsilon>0$ so that
\[
\frac 1M\asymp 1-\omega_M \leq \frac 14 b_0^{-0.9/\epsilon},
\]
which is roughly
\[
\epsilon\gg \frac 1{\log M}.
\]


\end{proof}



\section{Bounds for rich linear fractional transformations}
\label{sec:bounds-rich-linear}

We begin with some basic facts on subgroups and quotients of the group $GL_2(\F)$ of $2\times 2$ invertible matrices with entries in $\F$.
The \emph{special linear group}\/ $SL_2(\F)$ consists of elements of $GL_2(\F)$ with unit determinant.

The group $GL_2(\F)$ acts on the \emph{projective line}\/ $\prob^1(\F)$ by \emph{linear fractional transformations}.
Informally, $\prob^1(\F)=\F\cup\{\infty\}$ is the affine line $\F$ plus a point at infinity.
A linear fractional transformation is a map of the form
\begin{equation}
  \label{eq:3}
  x\mapsto \frac{ax+b}{cx+d},
\end{equation}
where
\begin{equation}
  \label{eq:3b}
    \begin{pmatrix}
    a & b\\
    c & d \\
  \end{pmatrix}
\end{equation}
is an element of $GL_2(\F)$.
If $x=\infty$, then $x\mapsto a/c$.
By abuse of notation, we may use the matrix in equation (\ref{eq:3b}) to denote the transformation in (\ref{eq:3}).

Clearly we may restrict the action (\ref{eq:3}) to $SL_2(\F)$.
A transformation acts trivially if and only if it is in the center $Z=\{\lambda I\colon \lambda\in\F^*\}$ of $GL_2(\F)$
The \emph{projective general linear group}\/ $PGL_2(\F)=GL_2(\F)/Z$ is the automorphism group of $\prob^1(\F)$ and the \emph{projective special linear group}\/ $PSL_2(\F)=SL_2(\F)/\{\pm I\}$ is a subgroup of $PGL_2(\F)$ \cite[Section 10.8]{borel1969linear}.
If every element of $\F^*$ is a square then $PSL_2(\F)=PGL_2(\F)$; otherwise, the index of $PSL_2(\F)$ in $PGL_2(\F)$ is 2.

The group $PGL_2(\F)$ acts \emph{simply 3-transitively}\/ on $\prob^1(\F)$, meaning that for every pair of triples $(x,y,z)$ and $(x',y',z')$ of distinct points in $\prob^1(\F)$, there is a unique transformation $g\in PGL_2(\F)$ such that
\[
g(x,y,z)=(x',y',z').
\]
The first proof of this for a general field $\F$ is due to Grothendieck, see \cite[Section 10.8]{borel1969linear}.
By a direct computation, one can show that $PSL_2(\F)$ acts doubly transitively on $\prob^1(\F)$.

The graphs of linear fractional transformations define hyperbolas in $\F\times \F$:
\begin{equation}
  \label{eq:45}
  y=\frac{ax+b}{cx+d} \iff cxy+ ax+ cy + d =0.
\end{equation}
If $g$ is the linear fractional transformation corresponding to the left-hand side of (\ref{eq:45}),
let $\Gamma_g$ denote the curve in $\F\times\F$ defined by $cxy+ax+cy+d=0$.
If $S\subseteq PSL_2(\F_p)$ and $Y\subseteq \F_p$, we may define the number of incidences between $P=Y\times Y$ and the set of hyperbolas $\Gamma_g$ with $g$ in $S$ by
\[
I(Y\times Y, S):=|\{(x,y,g)\in Y\times Y\times S\colon (x,y)\in \Gamma_g\}|.
\]
Note that
\[
I(Y\times Y,S)=\sum_{g\in S}|Y\cap gY|.
\]
The following theorem can be thought of as a bound for the number weighted incidences between a set of hyperbolas and a Cartesian product point set.
\begin{thm}
  \label{thm:15}
Let $\nu$ be a probability measure on $G=SL_2(\F_p)$ such that
\begin{enumerate}
\item $\norm{\nu}_\infty \leq K^{-1}$
\item for all $g\in G$ and all proper subgroups $\Gamma\leq G$, we have $\nu( g\Gamma)\leq K^{-1}$.
\end{enumerate}
Then for any set $Y\subseteq\prob^1(\F_p)$ and any element $z\in GL_2(\F_p)$, there are absolute constants  $c_*\in(0,1)$ and $C_*\geq 6$ such that
\[
\left|
\sum_g (\delta_z\ast\nu)(g)|Y\cap gY|-\frac{|Y|^2}p
\right|
\leq C_*|Y| p^{-\delta},
\]
where $\delta_z(x)=1$ if $x=z$ and $\delta_z(x)=0$ otherwise,
\[
k =\frac{3\log p}{c_*\log K},
\]
and \fxnote{Sigma-8: are we worried that we use $\delta$ as a constant and as the function $\delta_z$?}
\[
\delta=\delta(k)=\frac 1{2^{k+2}}.
\]
\end{thm}
As a corollary, we have the following incidence bound, originally proved by Bourgain \cite{bourgain2012modular} and used by Bourgain, Gamburd, and Sarnak to prove that a certain graph related to Markov triples is connected \cite{bourgain2016markoff,bourgain2016markoff-a}.
\begin{cor}[Bourgain's hyperbola incidence bound]
  \label{cor:4}
Given $Y\subseteq\prob^1(\F_p)$ and $S\subseteq G= PSL_2(\F_p)$ such that
\begin{itemize}
\item $|S|\geq p^\epsilon$, \fxnote{we can weaken this to $|S|\geq |Y|^\epsilon$ if we use Corollary~\ref{cor:8}}
\item for all $g\in G$ and all proper subgroups $\Gamma\leq G$ we have $|S\cap g\Gamma|\leq |S|^{1-\eta}$,
\end{itemize}
we have
\[
\left| I(Y\times Y, S) - \frac{|S||Y|^2}p\right| \leq C_*|Y||S|p^{-\delta},
\]
where $\delta = 2^{-(k+2)}$ and $k=3(c_*\eta\epsilon)^{-1}$.
\end{cor}
\begin{proof}
  Apply Theorem~\ref{thm:15} with $K=p^{\eta\cdot\epsilon}$.
\end{proof}

The following bound applies when we know more structural information about the set of linear fractional transformations.
\begin{thm}
  \label{thm:5}
There is an absolute constant $b_0>1$ such that the following holds for all $0<\alpha<1$, all sufficiently large primes $p\gg 1$, and all $0<\tau\leq 1/8$.

Let $B=\{1,\ldots,M\}$ and let $C=\{1,\ldots, N\}$.
Suppose that $5\leq\min(M,N)\leq p^\tau$.

Set
\[
S =
\left\{
  \begin{pmatrix}
    1 & -b \\
    c & 1-bc \\
  \end{pmatrix}
\colon
b\in B, c\in C
\right\}.
\]
Let $\delta=0.25 b_0^{-1/\tau}$ and
let $Y\subseteq\prob^1(\F_p)$ be a subset of size $1\leq |Y| \leq p^{1-\delta}$.

If $|Y\cap gY|\geq \alpha|Y|$ for all $g$ in $S$, then
\[
\min(M,N) \leq 2\left( \frac 2\alpha \right)^{\tau/\delta}+1.
\]
\end{thm}

\section{Proofs of bounds for sums of reciprocals}
\label{sec:proofs-bounds-sums}

In this section, we prove Theorems~\ref{thm:1} and \ref{thm:2} using the results from the previous section.

\subsection{Proof of Theorem~\ref{thm:1}}

Before proving Theorem~\ref{thm:1}, we state some classification results for the subgroups of $SL_2(\F_p)$, then state a key lemma, which states that the matrices relevant to Theorem~\ref{thm:1} do not concentrate in subgroups.

\subsubsection{Subgroups of $SL_2(\F_p)$}

Let $\B$ denote the \emph{standard Borel subgroup}\/ of $SL_2(\F_p)$:
\[
\B =
\left\{
  \begin{pmatrix}
    a & b\\
    0 & d\\
  \end{pmatrix}
\colon a,b,d\in\F_p, ad=1
\right\}.
\]
We use $\B'$ to denote the projection of $\B$ to $PSL_2(\F_p)$.

Dickson \cite{dickson1901theory,dickson1958linear} classified the subgroups of $SL_2(\F_p)$ and $PSL_2(\F_p)$, see \cite[Theorem 6.17, Theorem 6.25]{suzuki1982group}.
\begin{thm}[Dickson]
\label{thm:6}
Let $p\geq 5$ be a prime.
Every proper subgroup of $PSL_2(\F_p)$ is \emph{isomorphic}\/ to one of the following groups:
\begin{enumerate}
\item the dihedral groups of order $p\pm 1$ and their subgroups,
\item the  \emph{standard Borel subgroup}\/ $\B'$ of $PSL_2(\F_p)$ and its subgroups,
\item $A_4,S_4,A_5$.
\end{enumerate}

Further, every proper subgroup of $SL_2(\F_p)$ is \emph{isomorphic}\/ to one of the following groups:
\begin{enumerate}
\item the dihedral groups of order $2(p\pm 1)$  and their subgroups,
\item the dicyclic groups of order $4p,4(p\pm1)$ and their subgroups,
\item the \emph{standard Borel subgroup}\/ $\B$ of upper triangular matrices, and its subgroup,
\item a finite group of order at most 120.
\end{enumerate}
\end{thm}
Thus every proper subgroup of $PSL_2(\F_p)$ containing more than 60 elements is solvable.
See \cite[Section 3.6]{suzuki1982group} for a proof of the classification of subgroups of $SL_2(\F)$ when $\F$ is an arbitrary field of characteristic $p$.

\begin{lem}
  \label{lem:6}
Every cyclic subgroup of $SL_2(\F_p)$ is conjugate (by matrices in $SL_2(\F_p)$) to a subgroup of $\B$ or to a subgroup of the following form:
\[
K_\epsilon := \left\{
  \begin{pmatrix}
    x & \epsilon y\\
    y & x\\
  \end{pmatrix}
\colon x,y\in\F_p,  x^2-\epsilon y^2=1
\right\},
\]
where $\epsilon$ is a non-square.
\end{lem}
\begin{proof}
Suppose 
\[
g=
\begin{pmatrix}
  a & b\\
  c & d\\
\end{pmatrix}
\]
generates a cyclic subgroup $H$ of $SL_2(\F_p)$.
If $\tr(g)^2-4$ is a square over $\F_p$, then $g$ is conjugate (over $\F_p$) to a matrix of the form
\[
  \begin{pmatrix}
    a & 0 \\
    0 & a^{-1}\\
  \end{pmatrix}
\qquad\mbox{or}\qquad
\begin{pmatrix}
  1 & b \\
  0 & 1 \\
\end{pmatrix}.
\]
Now, if $H$ is a subgroup of $SL_2(\F_p)$ that is isomorphic to a subgroup of the standard Borel subgroup $\B$, then $H$ is conjugate to a subgroup of $\B$ by an element of $SL_2(\F_p)$ \cite[Proposition 16.6]{borel1969linear}.

Otherwise, if $\tr(g)^2-4$ is not a square, we can write
\begin{equation}
  \label{eq:47}
    \begin{pmatrix}
    1 & 0\\
    (d-a)/2b&1\\
  \end{pmatrix}^{-1}
  \begin{pmatrix}
    a & b\\
    c& d\\
  \end{pmatrix}
  \begin{pmatrix}
    1 & 0\\
    (d-a)/2b&1\\
  \end{pmatrix}
=
  \begin{pmatrix}
    x & y\\
    \epsilon y & x\\
  \end{pmatrix},
\end{equation}
where
\[
x=\frac{a+d}2,\qquad y=b,\andd \epsilon = \frac{(a+d)^2-4}{4b^2}.
\]
\end{proof}
See also \cite[Section 6, (6.3)]{suzuki1982group} and \cite[Section 5.2]{fulton1991representation}.

\subsubsection{Non-concentration in subgroups}
For subsets $B,C\subseteq\F_p$, let
\begin{equation}
  \label{eq:2}
S=  S_\rho= S_\rho(B,C):=
\left\{
  \begin{pmatrix}
    -\rho^{-1}c & -1+\rho^{-1}bc\\
    1 & -b\\
  \end{pmatrix}
\colon b\in B, c\in C
\right\}.
\end{equation}
Since $S/Z$ has the same cardinality as $S$, we may consider $S$ as a subset of $PSL_2(\F_p)$.
\begin{lem}
\label{lem:8}
Let $S=S_\rho(B,C)$ be defined as in \eqref{eq:2}.
	Then for any $g_1, g_2 \in PSL_2 (\F_p)$ one has 
	\begin{equation}\label{f:intersection-}
	|g_1 \B' g_2 \cap S| \le \max\{ |B|, |C| \} \,.
	\end{equation}
	In particular, if $B=C$, then 
	\begin{equation}\label{f:intersection}
	|g_1 \B g_2 \cap S| \le |S|^{1/2} \,.
	\end{equation}
	Moreover, 
	for any dihedral subgroup $\G$ one has
\begin{equation}\label{f:intersection+}
	|g_1 \G g_2 \cap S| \le 8. 
\end{equation}
	\label{l:intersection}
\end{lem}
\begin{proof}
We will consider $S$ as a subset of $SL_2(\F_p)$; projection to $PSL_2(\F_p)$ cannot increase the size of the intersection of $S$ with subgroups.

First, we consider the number of elements of $S$ that are contained in a coset of a Borel subgroup.
Since all Borel subgroups are conjugate to the standard Borel subgroup $\B$, we consider the equation
\[
  \begin{pmatrix}
    x & y \\
    z & w \\
  \end{pmatrix}
  \begin{pmatrix}
    r & q \\
    0 & r^{-1}\\
  \end{pmatrix}
=
  \begin{pmatrix}
    -\rho^{-1}c & -1+\rho^{-1}bc\\
    1 & -b\\
  \end{pmatrix}
  \begin{pmatrix}
    X & Y \\
    Z & W \\
  \end{pmatrix},
\]
with $xw-yz=XW-YZ=1$;
that is,
\begin{equation}
  \label{eq:21}
    \begin{pmatrix}
    xr & qx+y/r \\
    zr & qz+w/r \\
  \end{pmatrix}
=
\begin{pmatrix}
  -\rho^{-1}c(X-bZ)-Z&   -\rho^{-1}c(Y-bW)-W \\
  X-bZ & Y-bW\\
\end{pmatrix}.
\end{equation}

Either $x$ or $z$ is non-zero, since $xz-yz=1$.
Suppose that $z\not=0$.
Then
\[
xzr=-z(\rho^{-1}c(X-bZ)+Z)
\]
so substituting $zr=X-bZ$, we have
\[
x(X-bZ)=-z(\rho^{-1}c(X-bZ)+Z).
\]
Thus
\[
c = \frac{-bxZ +(xX+zZ)}{-z\rho^{-1}(X-bZ)}
=\frac{(-xZ)b+(xX+zZ)}{(zZ\rho^{-1})b+(-z\rho^{-1}X)}.
\]
Since
\[
\det
\begin{pmatrix}
  -xZ & xX+zX \\
  \rho^{-1}zZ & -\rho^{-1}zX\\
\end{pmatrix}
=-\rho^{-1}z^2Z^2,
\]
$c$ is determined uniquely by $b$ if $Z\not=0$.
On the other hand, if $Z=0$, then $X\not=0$ since $XW-YZ=1$, so $c=-\rho xr/X$ and there are at most $|B|$ solutions to \eqref{eq:21}.
If we assume $x\not=0$, then a similar situation occurs, and in general there are at most $\max(|B|,|C|)$ solutions to \eqref{eq:21}.

Next, we consider the number of elements of $S$ contained in a dihedral or dicyclic subgroup of $SL_2(\F_p)$.
The number of such elements is at most four times the number of elements contained in a cyclic subgroup of $SL_2(\F_p)$; by Lemma~\ref{lem:6}, every cyclic subgroup is conjugate either to a subgroup of the standard Borel subgroup $\B$, in which case the previous analysis applies, or to a subgroup of the form
\[
K_\epsilon := \left\{
  \begin{pmatrix}
    u & \epsilon v\\
    v & u\\
  \end{pmatrix}
\colon u,v\in\F_p,  u^2-\epsilon v^2=1
\right\},
\]
where $\epsilon$ generates $\F_p^*$.
Thus we consider the equation
\[
  \begin{pmatrix}
    x & y \\
    z & w \\
  \end{pmatrix}
  \begin{pmatrix}
   u & \epsilon v \\
    v & u \\
  \end{pmatrix}
=
  \begin{pmatrix}
    -\rho^{-1}c & -1+\rho^{-1}bc\\
    1 & -b\\
  \end{pmatrix}
  \begin{pmatrix}
    X & Y \\
    Z & W \\
  \end{pmatrix},
\]
with $xw-yz=XW-YZ=1$;
that is,
\begin{equation}
\label{eq:31}
    \begin{pmatrix}
      xu+yv & \epsilon xv+yu\\
      zu+wv & \epsilon zv + wu\\
  \end{pmatrix}
=
\begin{pmatrix}
  -\rho^{-1}c(X-bZ)-Z&   -\rho^{-1}c(Y-bW)-W \\
  X-bZ & Y-bW\\
\end{pmatrix}.
\end{equation}
From \eqref{eq:31} we derive
\begin{equation}
  \label{eq:33}
  -Z=(x+\rho^{-1}cz)u+(y+\rho^{-1}cw)v 
= Au+Bv.
\end{equation}
Since $xw-yz=1$, either $x+\rho^{-1}cz\not=0$ or $y+\rho^{-1}cw\not=0$.
Solving \eqref{eq:33} for $u$ or $v$ and substituting into $u^2-\epsilon v^2=1$ yields
\[
(Bv+Z)^2-\epsilon A^2v^2=A^2\implies (B^2-\epsilon A^2)v^2+2BZv+Z^2-A^2=0
\]
or
\[
B^2u^2-\epsilon(Au+Z)^2=B^2\implies (B^2-\epsilon A^2)u^2 - 2\epsilon AZu - (Z^2+B^2)=0.
\]
In both cases, the leading coefficient is non-zero since $\epsilon$ is not a square, so there at are most two solutions for $u$ or $v$; since $u$ and $v$ determine one another by \eqref{eq:33}, there are at most two pairs $(u,v)$ such that \eqref{eq:31} holds.

The pair $(u,v)$ determines the left-hand side of equation \eqref{eq:31}.
Since either $Z\not=0$ or $W\not=0$, and either $X-bZ\not=0$ or $Y-bW\not=0$, once $(u,v)$ is fixed, $b$ and $c$ are determined.
Thus there are at most $2$ elements of $g K_\epsilon g'$ contained in $S$, and hence at most 4 elements of a coset of a dihedral group contained in $S$ or at most 8 elements of a coset of a dicyclic group contained in $S$.
\end{proof}

\subsubsection{Proof of Theorem~\ref{thm:1}}

\begin{proof}[Proof of Theorem~\ref{thm:1}]
 Suppose $|A+B| + | \rho\cdot A^{-1}+C|\leq M|A|$.
Let $Y= (A+B)\cup (\rho\cdot A^{-1}+C)$, so that $|Y|\leq M|A|$.
Let
\[
g_{b,c}(x)=\frac{\rho}{x-b}+c
\]
and let $S$ denote the set of matrices in $GL_2(\F_p)$ corresponding to the linear fractional transformations $g_{b,c}$ with $b\in B$ and $c\in C$:
\[
S=
\left\{
  \begin{pmatrix}
    c & \rho-bc\\
    1 & -b\\
  \end{pmatrix}
\colon b\in B, c\in C\right\}.
\]
Let 
\[
z^{-1}=
\begin{pmatrix}
  -\rho^{-1} & 0\\
  0 & 1\\
\end{pmatrix}
\]
and let $S'=z^{-1}S$ so that $S'\subseteq SL_2(\F_p)$.

By Lemma~\ref{lem:8}, we have $|g\Gamma\cap S'|\leq \max(|B|,|C|)$ for any proper subgroup $\Gamma\leq SL_2(\F_p)$, assuming that $\max(|B|,|C|)\geq 4$, which holds for $p$ sufficiently large.
Let $K=\min(|B|,|C|)$ and let $\nu$ be the uniform measure on $S'$.
Then $\norm{\nu}_\infty=|S'|^{-1}\leq K^{-1}$ and
\[
\nu(g\Gamma)=\frac{|g\Gamma\cap S'|}{|S'|}\leq K^{-1}
\]
for all proper subgroups $\Gamma\leq SL_2(\F_p)$.
It follows from Theorem~\ref{thm:15} that
\begin{equation}
  \label{eq:57}
  \sum_g (\delta_z\ast \nu)(g)|Y\cap gY| \leq \frac{|Y|^2}p +C_*|Y|p^{-\delta},
\end{equation}
where
\[
\delta =\frac 14 b_0^{-\log_Kp}.
\]
Since $K=\min(|B|,|C|)=p^\epsilon$, we have $\delta\approx 2^{-O(1/\epsilon)}$.

On the other hand, for all $g\in zS'=S$ we have $|Y\cap gY|\geq |A|\geq \frac 1{M}|Y|$, hence
\begin{equation}
  \label{eq:58}
  \sum_g (\delta_z\ast\nu)(g)|Y\cap gY|=\frac 1{|S|}\sum_{g\in z\in S}|Y\cap gY|\geq\frac 1{M}|Y|.
\end{equation}
If $|Y|\leq p^{1-\delta}$, then equations \eqref{eq:57} and \eqref{eq:58} yields
\begin{equation}
  \label{eq:59}
  M\geq \frac 1{(C_*+1)}p^\delta,
\end{equation}
as claimed.
\end{proof}

\subsection{Proof of Theorem~\ref{thm:2}}

The proof of Theorem~\ref{thm:2} uses the same idea as the proof of Theorem~\ref{thm:1}, but we use Theorem~\ref{thm:5} in place of Theorem~\ref{thm:15}.
In particular, Theorem~\ref{thm:2} does not require the non-concentration results from the previous proof.
(However, the proof of Theorem~\ref{thm:5} does use non-concentration.)

\begin{proof}[Proof of Theorem~\ref{thm:2}]
Suppose $|A+B| + | A^{-1}+C|\leq \alpha^{-1}|A|$.
Let $Y= -(A+B)\cup (A^{-1}+C)^{-1}$, so that $|Y|\leq \alpha^{-1}|A|$.
Let
\[
\frac 1{g_{b,c}(x)}=\frac{-1}{x-b}+c
\]
and let $S$ denote the set of matrices in $SL_2(\F_p)$ corresponding to the linear fractional transformations $g_{b,c}$ with $b\in B$ and $c\in C$.
For each $g\in S$, we have $|Y\cap gY|\geq \alpha|A|$.

The elements of $S$ have the form
\[
  \begin{pmatrix}
     1 & -b \\
     c & 1-bc\\
  \end{pmatrix},
\]
where $b\in\{1,\ldots,M\}$ and $c\in\{1,\ldots, N\}$.

By Theorem~\ref{thm:5} we have
\[
\min(|B|,|C|)\leq 2 \left( \frac 2\alpha \right)^{\tau/\delta}+1,
\]
provided that $\min(|B|,|C|)\geq 11$.

It follows that
\[
\alpha^{-1}\geq\frac 12 \left( \frac{\min(|B|,|C|)}{2} \right)^{\delta/\tau},
\]
which implies
\[
|A+B|+|A^{-1}+C|\geq \frac{|A|}2 \left( \frac{\min(|B|,|C|)}{2} \right)^{\delta/\tau}.
\]
\end{proof}

\section{$\ell^2$-flattening/higher energies}
\label{sec:ell2-flatt-energ}

For a probability measure $\mu$ on $SL_2(\F_p)$, let $\mu^{(\ell)}$ denote the $\ell$-fold convolution of $\mu$ with itself; that is, $\mu^{(1)}=\mu$ and $\mu^{(\ell+1)}=\mu\ast \mu^{(\ell)}$.
The \emph{adjoint $\mu^{\sim}$}\/ of a finitely supported measure $\mu$ is defined by $\mu^\sim(x)=\mu(-x)$.

The following theorem combines the ``middle-game'' and ``end-game'' steps of Bourgain and Gamburd's proof of uniform expansion for Cayley graphs of $SL_2(\F_p)$ \cite{bourgain2008uniform}.
See \cite{yehudayoff2012proving} and \cite{sarnak2014notes} for an overview of the three steps of the proof of the main theorem from \cite{bourgain2008uniform}.
\begin{thm}
  \label{thm:8}
Let $\mu$ be a symmetric probability measure on $SL_2(\F_p)$ such that for some parameter $K\geq 1$ 
\begin{itemize}
\item $\mu(g\Gamma)\leq K^{-1}$ for any proper subgroup $\Gamma\leq SL_2(\F_p)$ and element $g\in SL_2(\F_p)$, and
\item $\norm{\mu}_\infty\leq K^{-1}$.
\end{itemize}
Then for any integer $k$
\[
\left|\norm{\mu^{(2^k)}}_2^2 - \frac 1{|SL_2(\F_p)|}\right| \leq C_*^k K^{-c_* k}
\]
where $c_*\in(0,1)$ and $C_*>1$ are absolute constants.
\end{thm}

Before proving Theorem~\ref{thm:8}, we state some preliminaries: a ``quasi-randomness'' bound for convolution on $PSL_2(\F_p)$, and results from arithmetic combinatorics.

The following bound is due to Gowers \cite{gowers2008quasirandom} and Babai, Nikolov, and Pyber \cite{babai2008product}.
\begin{thm}
  \label{thm:9}
Let $\mu$ be a probability measure on $PSL_2(\F_p)$ and let $f\colon PSL_2(\F_p)\to \CC$ have mean zero: $\sum_g f(g)=0$.
Then
\[
\norm{\mu\ast f}_2 \leq p\norm{\mu}_2\norm{f}_2.
\]
\end{thm}
\fxnote{need to add details... this isn't exactly the form stated in either of the quoted papers}

The following version of \bsg{} theorem can be found in \cite{murphy2017group} or derived from arguments in \cite{tao2008product}.
Recall that the \emph{multiplicative energy}\/ of a finite subset $A$ of a multiplicative group is defined by
\[
E(A):=|\{(a_1,a_2,a_3,a_4)\in A^4\colon a_1a_2=a_3a_4\}|.
\]
\begin{lem}[\bsg{} theorem]
  \label{lem:1}
If $A$ is a finite subset of a group $G$ and $E(A)\geq \zeta |A|^3$, then there exists a set $S\subseteq G$ and an element $a$ in $A$ such that $S\subseteq a^{-1}A$, $|S|\gg \zeta^C|A|$, and $|S^3|\ll \zeta^{-C}|S|$, where $C>0$ is an absolute constant.
\end{lem}
The following lemma allows us to reduce a statement about measures whose self-convolutions have large $\ell^2$ norm to a statement about multiplicative energy.
\begin{lem}[Weighted \bsg{} {\cite[Lemma 1.4.1]{tao2015expansion}}]
  \label{lem:7}
Let $\nu$ be a finitely supported function on a multiplicative group with $\norm{\nu}_1\ll 1$.
Suppose that $\norm{\nu\ast \nu}_2^2 \geq M^{-1}\norm{\nu}_2^2$ for some $M>1$.
Then there exists a set $A\subseteq\supp(\nu)$ such that
\begin{equation}
  \label{eq:60}
\frac 1{M\norm{\nu}_2^2}\ll |A|\ll \frac{M^2}{\norm{\nu}_2^{2}},
\end{equation}
\begin{equation}
  \label{eq:61}
  |\nu(g)|\gg \frac {\norm{\nu}_2^2}{M^2}
\end{equation}
for all $g\in A$, and
\begin{equation}
  \label{eq:62}
  E(A)\gg M^{-3}\norm{\nu}_2^{-6}\gg M^{-9}|A|^3.
\end{equation}
\end{lem}
We prove Lemma~\ref{lem:7} in the Appendix.

The final ingredient in the proof of Theorem~\ref{thm:8} is Helfgott's \emph{product theorem}\/ for $SL_2(\F_p)$ \cite{helfgott2008growth}.
We quote the version from \cite{kowalski2013explicit}.
\begin{thm}[Growth in $SL_2(\F_p)$]
  \label{thm:10}
For $p$ prime and $A\subseteq SL_2(\F_p)$, if $A$ generates $SL_2(\F_p)$, then either
\begin{equation}
  \label{eq:all-of-sl2}
  (A\cup A^{-1}\cup\{e\})^3=SL_2(\F_p)
\end{equation}
or
\begin{equation}
  \label{eq:growth}
  |(A\cup A^{-1}\cup\{e\})^3|\geq |A|^{1+\delta},
\end{equation}
where $\delta=\frac 1{3024}$.
\end{thm}
Since \cite[Equation (3.2)]{helfgott2015growth}
\[
|(A\cup A^{-1}\cup\{e\})^3|\leq \left( 3\frac{|A^3|}{|A|} \right)^3|A|,
\]
equation~\eqref{eq:growth} implies
\begin{equation}
  \label{eq:growth-2}
  |A^3|\geq\frac 13|A|^{1+\delta/3}.
\end{equation}
\fxnote{Say why this theorem applies to $PSL_2(\F_p)$?}

\begin{proof}[Proof of Theorem~\ref{thm:8}]
  Let $\gamma=|SL_2(\F_p)|^{-1}$ and set $f(x)=\mu(x)-\gamma$.
By induction, one can show that $f^{(n)}(x)=\mu^{(n)}(x)-\gamma$.
Thus
\[
\norm{f^{(n)}}_2^2=\norm{\mu^{(n)}}_2^2-\gamma.
\]
Thus to prove Theorem~\ref{thm:8}, it suffices to show that
\begin{equation}
  \label{eq:444}
  \norm{f^{(2\ell)}}_2^2\leq\frac 1M \norm{f^{(\ell)}}_2^2,
\end{equation}
where $M^{-1}=C_*K^{-c_*}$ and $\ell$ is a dyadic integer.

By Theorem~\ref{thm:9}
\begin{equation}
  \label{eq:5}
  \norm{f^{(2\ell)}}_2^2\leq p^2\norm{f^{(\ell)}}_2^4,
\end{equation}
so we may assume that
\begin{equation}
  \label{eq:6}
  \norm{f^{(\ell)}}_2^2\geq\frac 1{p^2M},
\end{equation}
otherwise we are done.

Let $r(x)=f^{(\ell)}(x)$.
Then
\begin{equation}
  \label{eq:43}
\norm{r}_1 \leq 1+\norm{\mu^{(\ell)}}_1 = 2.
\end{equation}
Thus if \eqref{eq:444} is false, then we may apply Lemma~\ref{lem:7} with $\nu=r$ to find a subset $P\subseteq \supp(r)$ such that 
\begin{equation}
  \label{eq:63}
  E(P)\gg M^{-9}|P|^3,
\end{equation}
\begin{equation}
  \label{eq:64}
M^{-1}\norm{r}_2^{-2}\ll  |P|\ll M^2\norm{r}_2^{-2},
\end{equation}
and
\begin{equation}
  \label{eq:65}
  |r(g)|\gg M^{-2}\norm{r}_2^2
\end{equation}
for all $g$ in $P$.

Equations \eqref{eq:6} and \eqref{eq:64} imply that
\begin{equation}
  \label{eq:49}
  |P|\ll M^2\norm{r}_2^{-2}\ll M^3p^2.
\end{equation}

We have a lower bound on $|r(x)|=|f^{(\ell)}(x)|=|\mu^{(\ell)}(x)-\gamma|$ and would like a lower bound on $|\mu^{(\ell)}(x)|$.
If $\mu^{(\ell)}(x)<2\gamma$, then $|r(x)|<\gamma$; however by (\ref{eq:6}) and (\ref{eq:65}) this implies that
\[
\frac 1{p^3-p}=\gamma > |r(x)| \gg \frac{\norm{r}_2^2}{M^2}\geq \frac 1{p^2M^3},
\]
hence $M\gg p^{1/3}$.
Choosing, say $M\leq p^{1/4}$, for $p$ sufficiently large, we have
\begin{equation}
  \label{eq:50}
  \mu^{(\ell)}(x)\gg M^{-2}\norm{r}_2^2
\end{equation}
for all $x$ in $P$.

Now we apply Lemma~\ref{lem:1} to $P$ to find a subset $S\subseteq g^{-1}P$ for some $g$ in $P$ such that $|S|\gg M^{-C}|P|$ and $|S^3|\ll M^{C}|S|$ for an absolute constant $C>0$.
By \eqref{eq:49} and $M\leq p^{1/4}$, we may apply Theorem~\ref{thm:10} with $\delta<1/4$ to find that either $S$ is contained in a proper subgroup $\Gamma\leq SL_2(\F_p)$ or \fxnote{Need to fix application of product theorem}
\begin{equation}
  \label{eq:51}
  |S|^{1+\epsilon}\ll_\delta |S^3|\ll M^{C}|S|,
\end{equation}
for some $\epsilon=\epsilon(\delta)>0$.
(We may assume $\delta$ is fixed, say $\delta=1/5$.)

We will choose our parameters so that \eqref{eq:51} cannot happen.
Equation~\eqref{eq:51} implies that
\[
M^{-C\epsilon}|P|^\epsilon\ll |S|^\epsilon\ll M^{C},
\]
hence
\begin{equation}
  \label{eq:54}
  |P|\ll M^{C(1+1/\epsilon)}.
\end{equation}

On the other hand, by the assumption $\norm{\mu}_\infty\leq K^{-1}$, we have
\begin{equation}
  \label{eq:52}
  \norm{r}_\infty\leq \gamma + \norm{\mu^{(\ell)}}_\infty \leq \gamma+\norm{\mu^{(\ell-1)}}_1\norm{\mu}_\infty \leq \frac 2K,
\end{equation}
since we may assume $K\leq\gamma^{-1}=p^3-p$.
Thus 
\begin{equation}
  \label{eq:70}
 \norm{r}_2^2\ll K^{-1}\norm{r}_1\ll K^{-1}.
\end{equation}

By equations \eqref{eq:64} and \eqref{eq:70}, it follows that
\begin{equation}
  \label{eq:53}
  |P| \gg M^{-1}K \gg K^{1-c_*}.
\end{equation}
Combining \eqref{eq:54} and \eqref{eq:53} yields a contradiction if $c_*$ is sufficiently small, depending on $\epsilon$ (hence on $\delta$).

Thus we may assume that \eqref{eq:51} does not hold, and hence $S$ is a contained in a proper subgroup $\Gamma\subseteq SL_2(\F_p)$.
Again, we will derive a contradiction.
Since $S\subseteq g^{-1}P$, we have 
\begin{equation}
  \label{eq:55}
  |g\Gamma\cap P|\geq |S|\gg M^{-C}|P|.
\end{equation}
By \eqref{eq:50} and \eqref{eq:64}, it follows that
\begin{equation}
  \label{eq:56}
  \mu^{(\ell)}(g\Gamma)\geq \frac{|g\Gamma\cap P|}{M^2\norm{r}_2^{-2}} \gg \frac{M^{-C}|P|}{M^3|P|}=M^{-C-3}.
\end{equation}
However, by assumption we have
\[
\mu^{(\ell)}(g\Gamma)\leq \norm{\mu^{(\ell-1)}}_1\sup_x \mu(xg\Gamma) \leq \frac 1K.
\]
If $c_*$ is sufficiently small, this contradicts \eqref{eq:56}.

It follows that \eqref{eq:444} must hold, and the proof is complete.
\end{proof}


\section{Proof of Theorem~\ref{thm:15}}
\label{sec:proof-theor-thm:15}

Now we prove Theorem~\ref{thm:15}, which we recall here.
\begin{repthm}{thm:15}
Let $\nu$ be a probability measure on $G=SL_2(\F_p)$ such that
\begin{enumerate}
\item $\norm{\nu}_\infty \leq K^{-1}$
\item for all $g\in G$ and all proper subgroups $\Gamma\leq G$, we have $\nu( g\Gamma)\leq K^{-1}$.
\end{enumerate}
Then for any set $Y\subseteq\prob^1(\F_p)$ and any element $z\in GL_2(\F_p)$
\[
\left|
\sum_g (\delta_z\ast\nu)(g)|Y\cap gY|-\frac{|Y|^2}p
\right|
\leq C_*|Y| p^{-\delta(k)},
\]
where \fxnote{$k$ should be an integer...}
\[
k =\frac{3\log p}{c_*\log K},
\]
and
\[
\delta(k)=\frac 1{2^{k+2}}
\]
for absolute constants $c_*\in(0,1)$ and $C_*\geq 6$.
\end{repthm}

The proof of Theorem~\ref{thm:15} requires pseudo-randomness bounds for the action of $PSL_2(\F_p)$ on the projective line $\prob^1(\F_p)$.

Let $G$ be a group acting on a set $X$, let $\mu\colon G\to\mathbb{C}$ and let $f\colon X\to\mathbb{C}$.
We define the convolution of $\mu$ and $f$ by \fxnote{move to notation section}
\[
\mu\ast f(y) := \sum_{gy=x}\mu(g)f(y).
\]
\begin{prop}
  \label{prop:1}
Suppose $G$ is a finite group that acts doubly transitively on a set $X$.
Suppose $\mu\colon G\to\CC$ and $f,h\colon X\to\CC$ satisfy $\sum_x f(x)=\sum_x h(x)=0$.
Then
\[
|\ip{\mu\ast f,h}|\leq \sqrt{\frac{|G|}{|X|-1}}\norm{\mu}_2\norm{f}_2\norm{h}_2.
\]
\end{prop}
We give an elementary proof of Proposition~\ref{prop:1} in the Appendix, but it also follows from a result of Gill \cite[Proposition 1.4]{gill2016quasirandom} and a lower bound on the degree of doubly-transitive permutation representations \cite[Exercise 2.6]{serre1977linear}.

Since $G=PSL_2(\F_p)$ acts doubly-transitively on $\prob^1(\F_p)$ we have the following bound.
\begin{cor}
  \label{cor:1b}
Let $\mu$ be a probability measure on $PSL_2(\F_p)$ and let $f$ be a function on $\prob^1(\F_p)$ with mean zero.
Then
\[
\norm{\mu\ast f}_2 \leq p \norm{\mu}_2\norm{f}_2.
\]
\end{cor}
\begin{proof}
  Apply Proposition~\ref{prop:1} with $h=\mu\ast f$.
\end{proof}

The following theorem is better than Corollary~\ref{cor:1b} when $|Y|$ is small.
\fxnote{We do not use this, but could add in a variation on Theorem~\ref{thm:1} that uses it}
A second bound is more useful if $|Y|$ is small.
\begin{thm}
  \label{thm:4}
Let $W$ be a subset of $\prob^1(\F_p)$ and let $\mu$ be a probability measure on $PSL_2(\F_p)$.
Then either $\gen{\mu\ast W, W}<4$ or
\begin{equation}
  \label{eq:34}
\gen{\mu\ast W, W}\leq 2\norm{\mu}_\infty^{1/3} \left( |W|-\frac{|W|}{p+1} \right)^2.
\end{equation}

If $f_W(x):=W(x)-|W|/(p+1)$ is the balanced function of $W$, then either $\gen{\mu\ast f_W,f_W}\leq 8$ or
\begin{equation}
  \label{eq:35}
  \gen{\mu\ast f_W,f_W}\leq 4 \norm{\mu}_\infty^{1/3} \left( |W|-\frac{|W|}{p+1} \right)^2.
\end{equation}
\end{thm}
We prove Theorem~\ref{thm:4} in the appendix.

\begin{proof}[Proof of Theorem~\ref{thm:15}]
For convenience, write $\nu_0=\delta_z\ast \nu$.
Let
\[
\sigma = \sum_x Y(x)(\nu_0\ast Y)(x)
\]  
and let $f_Y(x)=Y(x) - |Y|/(p+1)$, so that 
\[
\sigma = \frac{|Y|^2}{p+1} +\sum_x Y(x) (\nu_0\ast f_Y)(x) = \frac{|Y|^2}{p+1} +\sigma_*.
\]
By Cauchy-Schwarz, we have
\[
\left|
\sigma-\frac{|Y|^2}{p+1}
\right|^2=
\sigma_*^2 \leq |Y|\sum_x f_Y(x) (\nu_0^\sim\ast \nu_0\ast f_Y)(x).
\]
(Recall that $\nu_0^\sim(x)=\nu_0(x^{-1})$.)

Set $\eta(x)=(\nu_0^\sim\ast \nu_0)(x)=(\nu^\sim\ast\nu)(x)$, so that
\[
\sigma_*^2 \leq |Y|\sum_x f_Y(x) (\eta\ast f_Y)(x).
\]
Note that $\eta$ is a symmetric probability measure on $SL_2(\F_p)$ satisfying $\norm{\eta}_\infty\leq \norm{\nu}_\infty$ and
\begin{equation}
  \label{eq:29}
  \eta(g\Gamma)\leq \norm{\nu}_1 \max_x\nu(x^{-1}g\Gamma)\leq \frac 1{K}
\end{equation}
for all $g\in G$ and proper subgroups $\Gamma\leq G$.

Iterating Cauchy-Schwarz, we have
\begin{equation}
  \label{eq:14}
  \sigma_*^{2^{k+1}}\leq |Y|^{2^{k+1}-1}\sum_x f_Y(x) (\eta^{(2^k)}\ast f_Y)(x).
\end{equation}
Thus by Corollary~\ref{cor:1b} we have
\begin{equation}
  \label{eq:15}
  \sigma_*^{2^{k+1}}\leq p|Y|^{2^{k+1}-1}\norm{f_Y}_2^2 \norm{\eta^{(2^k)}}_2\leq p|Y|^{2^{k+1}}\norm{\eta^{(2^k)}}_2,
\end{equation}
since $\norm{f_Y}_2^2=|Y|-|Y|^2/(p+1)\leq |Y|$.

By (\ref{eq:29}), we may apply Theorem~\ref{thm:8} with $\mu=\eta$ to find
\begin{equation}
  \label{eq:16}
  \norm{\eta^{(2^k)}}_2^2\leq \frac 1{p^3-p} +\frac{C_*^k}{K^{c_*k}}.
\end{equation}

Choosing $k$ such that $K^{c_*k}=p^3$, we have
\begin{equation}
  \label{eq:17}
  \norm{\eta^{(2^k)}}_2\leq \frac {\sqrt{C_*^k+2}}{p^{3/2}}.
\end{equation}
Combining \eqref{eq:17} with \eqref{eq:15}, we have
\begin{equation}
  \label{eq:18}
  \sigma_*^{2^{k+1}}\leq (C_*^{k}+2)^{1/2}p^{-1/2}|Y|^{2^{k+1}}.
\end{equation}
Hence
\begin{equation}
  \label{eq:19}
  \sigma \leq \frac{|Y|^2}p +C_k \frac{|Y|}{p^{\delta}},
\end{equation}
where $\delta = 2^{-(k+2)}$ and
\begin{equation}
  \label{eq:20}
  C_k^{2^{k+1}} =\sqrt{C_*^{k}+2}.
\end{equation}
If $C_*\geq 6$, then a calculation shows that $C_k\leq C_*$.

Since $K^{c_*k}=p^3$, we have 
\[
k=\frac{3\log p}{c_*\log K}.
\]
\end{proof}

\begin{cor}
  \label{cor:8}
Let $\nu$ be a probability measure on $G=SL_2(\F_p)$ such that
\begin{enumerate}
\item $\norm{\nu}_\infty \leq K^{-1}$
\item for all $g\in G$ and all proper subgroups $\Gamma\leq G$, we have $\nu( g\Gamma)\leq K^{-1}$.
\end{enumerate}
Then for any set $Y\subseteq\prob^1(\F_p)$ and any element $z\in GL_2(\F_p)$
\[
\left|
\sum_g (\delta_z\ast\nu)(g)|Y\cap gY|-\frac{|Y|^2}p
\right|
\leq C_*|Y|^{1-\delta}
\]
where 
\[
  \delta= \frac 1{2^{k+1}} \left( \frac{c_*(k-1)\log K}{3\log |Y|}-1 \right),
\]
and $c_*\in(0,1)$ and $C_*\geq 6$ are absolute constants.
\end{cor}
\begin{proof}
  Starting from (\ref{eq:14}) and applying Theorem~\ref{thm:4}, we have
  \begin{equation}
    \label{eq:10}
    \sigma_*^{2^{k+1}}\leq 4|Y|^{2^{k+1}+1}\norm{\eta^{(2^k)}}_\infty^{1/3}.
  \end{equation}
By Cauchy-Schwarz and Theorem~\ref{thm:8}, recalling that $\gamma=1/(p^3-p)$, we have
\begin{equation}
  \label{eq:11}
  \norm{\eta^{(2^k)}}_\infty=  \norm{\eta^{(2^{k-1})}\ast \eta^{(2^{k-1})}}_\infty\leq \norm{\eta^{(2^{k-1})}}_2^2\leq\gamma+C_*^{k-1}K^{-c_*(k-1)}.
\end{equation}
Assuming $K^{c_*(k-1)}\leq p^3$, by (\ref{eq:10}) and (\ref{eq:11}), we have
\begin{equation}
  \label{eq:37}
  \sigma_*^{2^{k+1}}\leq 8C_*^{(k-1)/3}|Y|^{2^{k+1}+1}K^{-c_*(k-1)/3}.
\end{equation}
Choosing
\[
K=|Y|^{3(1+2^{k+1}\delta)/c_*(k-1)},
\]
we have
\begin{equation}
  \label{eq:41}
  \sigma_*\leq (8C_*^{(k-1)/3})^{2^{-k-1}}|Y|^{1-\delta},
\end{equation}
hence
\begin{equation}
  \label{eq:42}
  \left|\sum_xY(x)(\delta_z\ast\nu\ast Y)(x) - \frac{|Y|^2}{p+1} \right|\leq (8C_*^{(k-1)/3})^{2^{-k-1}}|Y|^{1-\delta},
\end{equation}
where
\begin{equation}
  \label{eq:44}
  \delta= \frac 1{2^{k+1}} \left( \frac{c_*(k-1)\log K}{3\log |Y|}-1 \right).
\end{equation}
For instance, if $k= 1 + 3\log|Y|/c_*\log K$, we have $\delta=2^{-k}$.
\end{proof}

\section{Proof of Theorem~\ref{thm:5}}
\label{sec:proof-theor-thm:5}

\subsection{Notation and statement of main lemmas}
\newcommand{\calG}{\mathcal{G}}

For a group $G$ and a finite subset $S\subseteq G$, the \emph{Cayley graph}\/ $\G=\cay(G,S)$ has vertex set $G$ and edges defined by $\{x,sx\}$ with $x\in G$ and $s\in S\cup S^{-1}$; it is $|S\cup S^{-1}|$-regular.
The \emph{girth}\/ of a graph is the length of its shortest cycle; we introduce a related quantity for Cayley graphs.
Let $d(\calG)$ be the smallest positive integer such any two distinct paths in $\calG$ of length $\leq d(\calG)$ starting at the identity have distinct end points.
Since $\calG$ is vertex-transitive, the girth of $\calG$ is either $2d(\calG)$ or $2d(\calG)-1$.
Hence a lower bound for $d(\calG)$ is equivalent to a lower bound for the girth of $\calG$.
\begin{thm}
  \label{thm:11}
For all $0<\alpha<1$ the following holds for all sufficiently large primes $p$:

Let $S\subseteq G=PSL_2(\F_p)$ be a set of transformations such that 
\begin{equation}
  \label{eq:4}
  d(\cay(G,S))\geq \tau_0\log_{|S|}(p)
\end{equation}
for some $\tau_0 > 0$.

Let $\delta=0.25\cdot b_0^{-1/\tau}$ where $0<\tau<\tau_0/2$ and $ b_0>1$ is an absolute constant.\fxnote{$b_0=2^{12/c_*}$ where $c_*$ is the constant from the flattening lemma}
Let $Y\subseteq\prob^1(\F_p)$ be a subset of size $1\leq |Y| \leq p^{1-\delta}$.

If $5\leq |S|\leq p^{\tau}$ and there is an element $z$ in $GL_2(\F_p)$ such that
\begin{equation}
  \label{eq:5}
  \sum_{g\in zS}|Y_g| \geq \alpha |Y||S|,
\end{equation}
then
\[
|S| \leq \left( \frac 2\alpha \right)^{\tau/\delta}.
\]
\end{thm}
The girth condition (\ref{eq:4}) is satisfied by random subsets (asymptotically almost surely) \cite{gamburd2007girth} and projections of generators of non-elementary subgroups of $SL_2(\Z)$ \cite{bourgain2008uniform}.

\begin{thm}
  \label{thm:3}
Let $N\geq 5$ be a positive integer and let $T$ denote the following set of elements of $PSL_2(\F_p)$:
\begin{equation}
  \label{eq:7}
  T=
\left\{
  \begin{pmatrix}
    1 & -2j \\
    2j & 1-4j^2\\
  \end{pmatrix}
\colon 1\leq j \leq N
\right\}.
\end{equation}
Then for all $p\gg 1$
\[
d(\cay(PSL_2(\F_p),T)) \geq \frac 14\log_N p.
\]
\end{thm}

Theorem~\ref{thm:5}, which we recall here, follows from Theorems~\ref{thm:11} and \ref{thm:3}.
\begin{repthm}{thm:5}
There is an absolute constant $b_0>1$ such that the following holds for all $0<\alpha<1$, all sufficiently large primes $p\gg 1$, and all $0<\tau\leq 1/8$.

Let $B=\{1,\ldots,M\}$ and let $C=\{1,\ldots, N\}$.
Suppose that $11\leq\min(M,N)\leq p^\tau$.

Set
\[
S =
\left\{
  \begin{pmatrix}
    1 & -b \\
    c & 1-bc \\
  \end{pmatrix}
\colon
b\in B, c\in C
\right\}.
\]
Let $\delta=0.25 b_0^{-1/\tau}$ and
let $Y\subseteq\prob^1(\F_p)$ be a subset of size $1\leq |Y| \leq p^{1-\delta}$.

If $|Y\cap gY|\geq \alpha|Y|$ for all $g$ in $S$, then
\[
\min(M,N) \leq 2\left( \frac 2\alpha \right)^{\tau/\delta}+1.
\]
\end{repthm}

\begin{proof}
Let $N_0=\floor{\min(M,N)/2}$.
If $T$ is defined as in \eqref{eq:7} with $N=N_0$, then $T\subseteq S$, so $|Y\cap gY|\geq\alpha |Y|$ for all $g$ in $T$.
By Theorem~\ref{thm:3}, we have $d(\cay(G,T))\geq\frac 14\log_{|T|}p$, thus by Theorem~\ref{thm:11} with $\tau_0=1/4$, we have 
\[
\min(\floor{M/2},\floor{N/2}) \leq \left( \frac 2\alpha \right)^{\tau/\delta}
\]
for all $0<\tau<1/8$, provided that $5\leq |T|\leq p^{\tau}$.
\end{proof}

\subsection{Proof of Theorem~\ref{thm:11}}

Throughout this section, $G=PSL_2(\F_p)$, $S$ is a subset of $G$, and $k=|S|$.

The girth bound (\ref{eq:4}) implies that the products $S^m$ of $S$ grow as quickly as possible for $m\leq \gamma:=d(\cay(G,S))$.
The following lemma is an immediate consequence of the definition of $\gamma$.
\begin{lem}[Girth bound implies locally free]
  \label{lem:2}
For $m\leq\gamma$, the ball of radius $m$ about the identity in $\cay(G,S)$ is isomorphic to the ball of radius $m$ about the identity in the Cayley graph of the free group $F_k$ on $k$ generators.
\end{lem}

If $\gamma \geq 2$, then $S\cap S^{-1}=\emptyset$, so $\mu(x)=\frac 1{2k}1_{S\cup S^{-1}}(x)$ is the uniform measure on $S\cup S^{-1}$.
Recall that the $m$-fold convolution of $\mu$ with itself is defined by
\[
\mu^{(m)}(x) = \sum_{y_1\cdots y_m=x}\mu(y_1)\cdots \mu(y_m).
\]
For $m\geq 1$, the measure $\mu^{(m)}$ is a symmetric probability measure on $G$.
\begin{lem}[Bounds for convolutions of the uniform measure on $S$]
  \label{lem:3}
For $\gamma\geq 2$ and  $m\leq\gamma$, we have
\begin{equation}
  \label{eq:8}
  \sum_{g\in G} |\mu^{(m)}(g)|^2 \leq \left( \frac 2k \right)^{m}.
\end{equation}
\end{lem}
\begin{proof}
The claimed bound is trivial if $k=1$, so without loss of generality, assume that $k\geq 2$.

By Lemma~\ref{lem:2}, when $m\leq\gamma$, $\mu^{(m)}(x)$ is equal to the probability $p^{(m)}(e,x)$ of arriving at $x$ after $m$ steps from the identity in the uniform random walk on $F_k$; see \cite[p. 637]{bourgain2008uniform}.
(By abuse of notation, we will use $x$ to denote an element of $F_k$ as well as the corresponding element in the ball of radius $m$ about the identity in $G$.)

Since $\mu$ is symmetric, we have
\[
\mu^{(m)}(x)=\mu^{(m)}(x^{-1})=p^{(m)}(e,x^{-1})=p^{(m)}(x,e).
\]
Thus the probability of return to the identity in $2m$ steps is
\[
p^{(2m)}(e,e)=\sum_{x\in G}p^{(m)}(e,x)p^{(m)}(x,e) = \sum_{x\in G} |\mu^{(m)}(x)|^2.
\]
By \cite[Lemma 1.9]{woess2000random}, $p^{(2m)}(e,e)\leq\rho^{2m}$, where $\rho$ is the spectral radius of $F_k$.
Kesten~\cite{kesten1959symmetric} proved that if $k\geq 2$ then
\[
\rho\leq \left( \frac{2k-1}{k^2} \right)^{1/2}\leq \left( \frac 2k \right)^{1/2},
\]
which completes the proof.
\end{proof}

\begin{lem}[Non-concentration in proper subgroups]
  \label{lem:4}
Let $H$ be a proper subgroup of $G$ and let $g$ be an element of $G$.

For $2\leq m\leq \gamma/2$ we have
\[
|\supp(\mu^{(m)})\cap gH| \leq m^6.
\]
\end{lem}
\begin{proof}
  If $m\leq \gamma/2$, then the support $S$ of $\mu^{(m)}$ is isomorphic to the ball $B_{m/2}$ of radius $m/2$ in $F_k$, hence $S^{-1}S$ is isomorphic to $B_m$.

Since $|S^{-1}\cap Hg^{-1}|=|S\cap gH|$, in particular, $S^{-1}\cap Hg^{-1}$ is non-empty (otherwise we are done), hence
\[
|S\cap gH|\leq |(S^{-1}\cap Hg^{-1})(S\cap gH)| \leq |S^{-1}S\cap H|.
\]

By Theorem~\ref{thm:6}, if $|H|>60$ then $H$ is two step solvable, hence
\begin{equation}
\label{eq:9}
  [[g_1,g_2],[g_3,g_4]]=1
\end{equation}
for all $g_1,\ldots,g_4$ in $H$.
By \cite[Proposition 8]{bourgain2008uniform}, the number of elements in $B_m$ satisfying (\ref{eq:9}) is at most $m^6$.

If $|H|\leq 60$, then the bound still holds, since $m^6\geq 64$.
\end{proof}

\begin{proof}[Proof of Theorem~\ref{thm:11}]
  Let $\mu$ be the uniform measure on $S\cup S^{-1}$.
The hypothesis (\ref{eq:5}) translates to
\[
\sum_g (\delta_z\ast \mu)(g)|Y_g|\geq \frac{\alpha|Y||S|}{|S\cup S^{-1}|}\geq \frac{\alpha}2 |Y|.
\]
By Cauchy-Schwarz and the inclusion
\[
Y_g\cap Y_{g'}= Y\cap gY\cap g'Y\subseteq gY\cap g'Y=gY_{g^{-1}g'},
\]
we have
\begin{align*}
  \left( \frac\alpha2 \right)^2|Y| &\leq \sum_{g,g'}(\delta_z\ast \mu)(g)(\delta_z\ast \mu)(g')|Y_g\cap Y_{g'}|\\
&\leq \sum_g ((\delta_z\ast \mu)^\sim\ast(\delta_z\ast \mu))(g)|Y\cap Y_{g}|
= \sum_g \mu^{(2)}(g)|Y_g|,
\end{align*}
where $f^\sim(x):=f(-x)$ is the adjoint of function $f$.

Iterating this, we find that
\begin{equation}
  \label{eq:24}
\left( \frac\alpha2 \right)^{2^j}|Y| \leq \sum_g \mu^{(2^j)}(g)|Y_g|.
\end{equation}
Let $m$ denote a dyadic integer less than or equal to $\gamma/2$.
(Recall $\gamma = d(\cay(G,S))\geq \tau_0\log_{|S|}p$.)
We will choose $m$ presently.

Let $\nu=\mu^{(m)}$.
By Lemma~\ref{lem:3},
\begin{equation}
  \label{eq:22}
  \norm{\nu}_\infty \leq \norm{\mu^{(m/2)}}_2^2\leq \left( \frac 2{|S|} \right)^{m/2}.
\end{equation}
If $g\in G$ and $\Gamma$ is a proper subgroup of $G$ then by Lemma~\ref{lem:4}
\begin{equation}
  \label{eq:23}
  \nu(g\Gamma) \leq |g\Gamma\cap\supp(\mu^{(m)})|\norm{\nu}_\infty \leq m^6 \left( \frac 2{|S|} \right)^{m/2}.
\end{equation}

Define $K^{-1}= |S|^{-m/4}$ and define $\tau$ by $m=\tau\log_{|S|}p$.
We want $K^{-1}\geq m^6(2/|S|)^{m/2}$, so that the hypotheses of Theorem~\ref{thm:15} are satisfied.
Thus we need
\begin{equation}
  \label{eq:12}
  \frac 1{|S|^{m/4}}\geq m^6 \left( \frac 2{|S|} \right)^{m/2}\quad\mbox{or}\quad |S|^{m/2}\geq m^{12} 2^{m}.
\end{equation}
By the definition of $\tau$, \eqref{eq:12} is equivalent to
\begin{equation}
  \label{eq:13}
  p^{\tau/2} \geq (\tau\log_{|S|}p)^{12} p^{\tau\log_{|S|} 2} \quad\mbox{or}\quad p^{\tau(1-\log_{|S|}4)} \geq (\tau\log_{|S|}p)^{24}.
\end{equation}
If $|S|\geq 5$, then \eqref{eq:13} is satisfied for $p\gg 1$.

Since
\[
\norm{\nu}_\infty, \nu(g\Gamma) \leq  m^6 \left( \frac 2{|S|} \right)^{m/2} \leq \frac 1K,
\]
by Theorem~\ref{thm:15} we have
\begin{equation}
  \label{eq:25}
  \left( \frac\alpha2 \right)^m \leq\frac 1{|Y|}\sum_g \nu(g)|Y\cap gY| \leq \frac{|Y|}p + C_*  p^{-\delta},
\end{equation}
where $\delta=2^{-(k+2)}$ and $k=3\log p /(c_*\log K)$.

By the definition of $\tau=m/\log_{|S|}p$ we have
\begin{equation}
  \label{eq:26}
  k=\frac{3\log p}{c_*\log K}= \frac{12\log p}{c_* m\log |S|} = \frac {12}{c_*}\,\tau^{-1}.
\end{equation}

Suppose that $|Y|\leq p^{1-\delta}$.
Then
\begin{equation}
  \label{eq:27}
  \left( \frac\alpha2 \right)^m\leq \frac{C_*+1}{p^\delta}.
\end{equation}

Since
\[
\left( \frac 2\alpha \right)^m = p^{\tau\log_{|S|}(2\alpha^{-1})},
\]
equation \eqref{eq:27} implies $\delta\leq \tau\log_{|S|}(2\alpha^{-1})$, otherwise there is a contradiction for large $p$, since $C_*+1$ is an absolute constant.
Thus we have
\begin{equation}
  \label{eq:28}
\log|S| \leq \frac\tau\delta \log (2\alpha^{-1})\implies |S| \leq \left( \frac 2\alpha \right)^{\tau/\delta}.
\end{equation}

By \eqref{eq:26}, we have \[\delta=2^{-(k+2)}=\frac 14 b_0^{-1/\tau}\] for  $b_0=2^{12/c_*}$.
Since $m\leq \gamma/2$, we can take $0<\tau\leq \tau_0/2$.
Finally, since $m\geq 1$, we need $1\leq \tau\log_{|S|}p$, which follows from $|S|\leq p^\tau$.
\end{proof}

\subsection{Proof of Theorem~\ref{thm:3}}

Given a matrix
\[
g=
\begin{pmatrix}
  a & b \\
  c & d \\
\end{pmatrix}
\]
in $SL_2(\Z)$, we use $\norm g$ to denote its norm as an operator on $\ell^2(\R^2)$:
\[
\norm g := \sup_{|x|_2=1}|gx|_2,
\]
where $|(x_1,x_2)|_2=\sqrt{x_1^2+x_2^2}$.
For a finite collection of matrices $S\subseteq SL_2(\Z)$, we define
\[
n(S):= \max_{g\in S}\norm g.
\]
If $S\subseteq PSL_2(\Z)$, we define $n(S)=n(S')$, where $S'\subseteq SL_2(\Z)$ is some collection of matrices representing the elements of $S$.
Since $\norm g = \norm{-g}$, this is well defined.

If $S\subseteq PSL_2(\F_p)$, we define
\[
n(S) := \min\{n(\tilde S)\colon \tilde S\subseteq PSL_2(\Z), \tilde S \equiv S\mod p\}.
\]

We will use the notation $\tilde G = PSL_2(\Z)$ and $\tilde S$ for subsets of $\tilde G$; the map $\phi_p\colon \tilde G\to G=PSL_2(\F_p)$ is defined by reduction of the entries of matrices representing elements of $\tilde G$ modulo $p$.
Thus $S=\phi_p(\tilde S)$ in the above definition of $n(S)$.

A direct computation shows that
\begin{equation}
  \label{eq:30}
   \frac 12\sqrt{a^2+b^2+c^2+d^2}\leq
\Norm{
  \begin{pmatrix}
    a & b\\
    c & d\\
  \end{pmatrix}
}
\leq \sqrt{a^2+b^2+c^2+d^2},
\end{equation}
thus $|S| \ll n(S)^4$.


The following theorem of Margulis \cite[Section 6]{margulis1982explicit} gives a lower bound for $d(\calG)$ (and hence the girth) of the Cayley graph $\calG=\cay(G,S)$ in terms of the norm of $S$.
See also \cite[Section 2]{gamburd2002spectral}.
\begin{thm}[Girth bound for projections of free groups]
  \label{thm:13}
If the group $\Lambda$ generated by $\tilde S\subseteq\tilde G$ is free, then 
\[
d(\cay(\phi_p(\Lambda),\phi_p(\tilde S))) \geq \log_{n(\tilde S)} \left( \frac p2 \right).
\]
Hence
\[
\girth(\cay(\phi_p(\Lambda),\phi_p(\tilde S))) \geq 2\log_{n(\tilde S)} \left( \frac p2 \right)-1.
\]
\end{thm}

Let $F_2=\gen{a,b}$ be the free group on two generators $a$ and $b$.
In general, let $F_n$ denote the free group on $n$ generators; we say that $n$ is the \emph{rank}\/ of $F_n$.
If $S$ is a set of elements in a group that generates a free group $F_n$ with $n=|S|$, we say that $S$ \emph{freely generates}\/ $F_n$.
\begin{thm}
  \label{thm:111}
For $n\geq 1$, let $S_n=\{ab, a^2b^2,\ldots,a^nb^n\}\subseteq F_2$.
Then $S_n$ freely generates a subgroup of $F_2$ isomorphic to $F_n$.
\end{thm}
\begin{proof}
This is Exercise 12 in Section 1.4 of \cite{magnus2004combinatorial}.   
\end{proof}

The free group $F_2$ is relevant to our problem because it is a subgroup of $PSL_2(\Z)$.
Let $\Gamma(2)\leq PSL_2(\Z)$ be the kernel of the homomorphism defined by reduction $\mod 2$:
\[
\Gamma(2)=\left\{
  \begin{pmatrix}
    a & b\\
    c & d\\
  \end{pmatrix}\in PSL_2(\Z)\colon a, d\equiv 1\mod 2, b, c\equiv 0\mod 2
\right\}.
\]
It is known \cite{magnus2004combinatorial} that $\Gamma(2)$ contains an index two free subgroup $\Lambda$ on two generators $u$ and $v$ given by
\begin{equation}
  \label{eq:1}
  u=
\begin{pmatrix}
  1 & 2 \\
  0 & 1 \\
\end{pmatrix}
\andd
v=
\begin{pmatrix}
  1 & 0 \\
  2 & 1 \\
\end{pmatrix}.
\end{equation}

Let
\[
T=\{v^{-j}u^j\colon 1\leq j\leq N\}.
\]
It follows from Theorem~\ref{thm:111} that $T$ generates a free subgroup of $\Lambda$ of rank $N$.
\begin{cor}
  \label{cor:5}
If
\[
T=\left\{
  \begin{pmatrix}
    1 & -2j \\
    2j & 1-4j^2\\
  \end{pmatrix}
\colon 1\leq j\leq N
\right\},
\]
then $T$ generates a free subgroup of $PSL_2(\Z)$ of rank $|T|=N$.
\end{cor}
\begin{proof}
  Since
\[
  \begin{pmatrix}
    1 & -2j \\
    2j & 1-4j^2\\
  \end{pmatrix}
= 
\begin{pmatrix}
  1 & 0 \\
  2j & 0\\
\end{pmatrix}
\begin{pmatrix}
  1 & -2j \\
  0 & 1 \\
\end{pmatrix}
=v^ju^{-j}
\]
we have $T=\{v^ju^{-j}\colon 1\leq j\leq N\}$, where $u$ and $v$ are the matrices in \eqref{eq:1} that generate a subgroup of $PSL_2(\Z)$ isomorphic to $F_2$.
Since $v$ and $u^{-1}$ also generate the same subgroup, it follows by Theorem~\ref{thm:111} that $T$ generates a free subgroup of rank $|T|$.
(This is because $T$ is the set $S_N$ from Theorem~\ref{thm:111} with $u$ replaced by $u^{-1}$.)
\end{proof}

\begin{proof}[Proof of Theorem~\ref{thm:3}]
Let $\tilde T\subseteq PSL_2(\Z)$ be such that $\phi_p(\tilde T)=T$; we may take $\tilde T$ to be the same set of matrices in $T$, but with coefficients in $\Z$ instead of $\Z/p\Z$.
By Corollary~\ref{cor:5}, $\tilde T$ generates a free subgroup $\tilde \Lambda$ of $PSL_2(\Z)$ of rank $|T|$, so by Theorem~\ref{thm:13}
\[
d(\cay(\phi_p(\tilde\Lambda)),\phi_p(\tilde T)) \geq \log_{n(\tilde T)} \left( \frac p2 \right).
\]
By \eqref{eq:30}, we have
\[
\Norm{
  \begin{pmatrix}
    1 & -2j \\
    2j & 1-4j^2 \\
  \end{pmatrix}
}
\leq \sqrt{2+16j^4}\leq 5j^2,
\]
so $n(\tilde T)\leq 5|T|^2=5N^2\leq N^3$ if $N\geq 5$.
Thus
\[
d(\cay(\phi_p(\tilde \Lambda)),\phi_p(\tilde T)) \geq \frac 13\log_{N} \left( \frac p2 \right),
\]
which proves the claimed bound if $p$ is sufficiently large.\fxnote{need to ensure that $\phi_p(\tilde \Lambda)=PSL_2(\F_p)$?}
\end{proof}

\appendix
\section{Analytic lemmas}
\label{sec:analytic-lemmas}

In this Appendix, we prove some technical lemmas quoted above.

\subsection{Proof of Lemma~\ref{lem:7}}

Recall Lemma~\ref{lem:7}.
\begin{replem}{lem:7}
  Let $\nu$ be a finitely supported function on a multiplicative group with $\norm{\nu}_1\ll 1$.
Suppose that $\norm{\nu\ast \nu}_2^2 \geq M^{-1}\norm{\nu}_2^2$ for some $M>1$.
Then there exists a set $A\subseteq\supp(\nu)$ such that
\begin{equation*}
\tag{\ref{eq:60}}
\frac 1{M\norm{\nu}_2^2}\ll |A|\ll \frac{M^2}{\norm{\nu}_2^{2}},
\end{equation*}
\begin{equation*}
  |\nu(g)|\gg \frac {\norm{\nu}_2^2}{M^2}  \tag{\ref{eq:61}}
\end{equation*}
for all $g\in A$, and
\begin{equation*}
  E(A)\gg M^{-3}\norm{\nu}_2^{-6}\gg M^{-9}|A|^3.   \tag{\ref{eq:62}}
\end{equation*}
\end{replem}
The proof of Lemma~\ref{lem:7} follows the proof of Lemma 1.4.1 in \cite{tao2015expansion}.
\begin{proof}[Proof of Lemma~\ref{lem:7}]
  Suppose $G$ is a group, $\nu\colon G\to\CC$ has finite support, $\norm{\nu}_1\ll 1$, and
\[
\norm{\nu}_2^2\geq \frac 1M\norm{\nu}_2^2.
\]
We wish to find a subset $A\subseteq\supp(\nu)$ with $|A|\ll 1/\norm{\nu}_2^2$ and $|\nu(x)|\gg \norm{\nu}_2^2$ for all $x\in A$ such that $A$ has large additive energy.

Without loss of generality, we may replace $\nu$ by its absolute value, so we will assume that $\nu$ is non-negative.

Write $\nu=\nu_1+\nu_2+\nu_3$ where
\[
\nu_1 := \nu\cdot 1_{\{x\colon \nu(x)< \lambda\norm{\nu}_2^2\}},
\]
\[
\nu_3 := \nu\cdot 1_{\{x\colon \nu(x)> \Lambda\norm{\nu}_2^2\}},
\]
and
\[
\nu_2:=\nu-\nu_1-\nu_3.
\]
We want a lower bound for $\norm{\nu_2\ast\nu_2}_2^2$.

We have
\[
\norm{\nu_1}_2^2\leq \lambda\norm{\nu}_2^2\norm{\nu_1}_1 \ll \lambda\norm{\nu}_2^2
\]
and
\[
\norm{\nu_3}_1 \leq \frac 1{\Lambda\norm{\nu}_2^2}\norm{\nu_3}_2^2 \leq \frac 1\Lambda.
\]
By Young's inequality,
\[
\norm{\nu_1\ast \nu}_2 \leq \norm{\nu_1}_2\norm{\nu}_1\ll \lambda^{1/2}\norm{\nu}_2
\]
and
\[
\norm{\nu_3\ast \nu}_2 \leq \norm{\nu_3}_1\norm{\nu}_2\leq \frac 1{\Lambda}\norm{\nu}_2.
\]

It follows that
\begin{align*}
  \norm{\nu_2\ast\nu_2-\nu\ast\nu}_2&\leq \norm{\nu_2\ast (\nu_1+\nu_3)}_2+\norm{\nu\ast (\nu_1+\nu_3)}_2 \\
&\leq 2\norm{\nu\ast\nu_1}_2 + 2\norm{\nu\ast \nu_3}_2 \ll (\lambda^{1/2}+\Lambda^{-1})\norm{\nu}_2.
\end{align*}
Choosing $\lambda\approx 1/M$ and $\Lambda\approx M^{1/2}$, we have
\[
\norm{\nu_2\ast\nu_2}_2 \geq \norm{\nu\ast\nu}_2 - \norm{\nu_2\ast\nu_2-\nu\ast\nu}_2 \gg \frac 1{M^{1/2}}\norm{\nu}_2.
\]

Let $A:=\{x\colon \nu(x)\geq \lambda\norm{\nu}_2^2\}$.
Then
\[
\norm{A\ast A}_2 \geq \frac 1{\Lambda^2\norm{\nu}_2^4}\norm{\nu_2\ast \nu_2}_2\gg \frac 1{M^{3/2}\norm{\nu}_2^3},
\]
hence
\[
E(A) \gg \frac 1{M^{3}\norm{\nu}_2^6}.
\]

On the other hand, by Markov's inequality and $\norm{\nu}_1\ll 1$,
\[
|A| \ll \frac 1{\lambda^2\norm{\nu}_2^2} \ll \frac{M^2}{\norm{\nu}_2^2},
\]
so
\[
E(A)\gg \frac{|A|^3}{M^{9}}.
\]

The lower bound on $|A|$ in Equation~\eqref{eq:60} follows from
\[
|A|^3\geq E(A) \gg \frac 1{M^{3}\norm{\nu}_2^6}.
\]
\end{proof}

\subsection{Proof of Proposition~\ref{prop:1}}

\begin{repprop}{prop:1}
Suppose $G$ is a finite group that acts doubly transitively on a set $X$.
Suppose $\mu\colon G\to\CC$ and $f,h\colon X\to\CC$ satisfy $\sum_x f(x)=\sum_x h(x)=0$.
Then
\[
|\ip{\mu\ast f,h}|\leq \sqrt{\frac{|G|}{|X|-1}}\norm{\mu}_2\norm{f}_2\norm{h}_2.
\]  
\end{repprop}
The proof of Proposition~\ref{prop:1} is a completion argument, similar to the arguments in \cite{murphy2016point-line}.
\begin{proof}[Proof of Proposition~\ref{prop:1}]
The proof is a completion argument using Cauchy-Schwarz:
  \begin{align*}
      |\ip{\mu\ast f,h}|&\leq\sum_{g\in G}|\mu(g)| \left|\sum_{x\in X}f(g^{-1}x)\overline{h(x)} \right|\\
&\leq \norm{\mu}_2 \left(\sum_{g\in G}\left|\sum_{x\in X}f(g^{-1}x)\overline{h(x)} \right|^2 \right)^{1/2}.
  \end{align*}
Since $G$ acts transitively on $X$ and non-diagonal pairs in $X\times X$, we have
\begin{align*}
\sum_{g\in G}\left|\sum_{x\in X}f(g^{-1}x)\overline{h(x)} \right|^2&=
  \sum_{g\in G}\sum_{x,y\in X}f(g^{-1}x)\overline{f(g^{-1}y)}\overline{h(x)}h(y)\\
&=\sum_{g\in G}\sum_{x\in X}|f(g^{-1}x)|^2|h(x)|^2 + \sum_{g\in G}\sum_{x\not= y\in X}f(g^{-1}x)\overline{f(g^{-1}y)}\overline{h(x)}h(y) \\
&=\frac{|G|}{|X|}\sum_{x',x\in X}|f(x')|^2|h(x)|^2 + \frac{|G|}{|X|(|X|-1)}\sum_{x\not=y,x'\not=y'}f(x')\overline{f(y')}\overline{h(x)}h(y)\\
&=\frac{|G|}{|X|}\norm{f}_2^2\norm{h}_2^2 + \frac{|G|}{|X|(|X|-1)} \left( \left|\sum_x f(x)\right|^2-\norm{f}_2^2 \right) \left( \left|\sum_x h(x)\right|^2-\norm{h}_2^2 \right)\\
&=\frac{|G|}{|X|}\norm{f}_2^2\norm{h}_2^2 + \frac{|G|}{|X|(|X|-1)}\norm{f}_2^2\norm{h}_2^2\\
&=\frac{|G|}{|X|-1}\norm{f}_2^2\norm{h}_2^2.
\end{align*}
\end{proof}

\subsection{Proof of Theorem~\ref{thm:4}}

\begin{repthm}{thm:4}
Let $W$ be a subset of $\prob^1(\F_p)$ and let $\mu$ be a probability measure on $PSL_2(\F_p)$.
Then either $\gen{\mu\ast W, W}<4$ or
\begin{equation}
\gen{\mu\ast W, W}\leq 2\norm{\mu}_\infty^{1/3}|W|^2.
\end{equation}

If $f_W(x):=W(x)-|W|/(p+1)$ is the balanced function of $W$, then either $\gen{\mu\ast f_W,f_W}\leq 8$ or
\begin{equation}
  \gen{\mu\ast f_W,f_W}\leq 4 \norm{\mu}_\infty^{1/3} \left( |W|-\frac{|W|}{p+1} \right)^2.
\end{equation}  
\end{repthm}
\begin{proof}[Proof of Theorem~\ref{thm:4}]
  Since $PGL_2(\F_p)$ acts simply 3-transitively on $\prob^1(\F_p)$ and $PSL_2(\F_p)\leq PGL_2(\F_p)$, for any pair of distinct triples of points $(x_1,y_1,z_1), (x_2,y_2,z_2)$ in $\prob^1(\F_p)^3$ there is at most one element $g\in PSL_2(\F_p)$ such that
  \begin{equation}
    \label{eq:36}
    g (x_1,y_1,z_1)=(x_2,y_2,z_2).
  \end{equation}

Since the function $x\mapsto{x\choose 3}$ is convex and
\[
\gen{\mu\ast W,W}=\sum_g \mu(g)\gen{\delta_g\ast W,W},
\]
we have
\begin{align*}
  3!  {\gen{\mu\ast W,W}\choose 3}
&=3!{\sum_g \mu(g)\gen{\delta_g\ast W,W}\choose 3}\\
&\leq 3!\sum_g \mu(g) {\gen{\delta_g\ast W,W}\choose 3}\\
&\leq 3!\norm{\mu}_\infty\sum_g {\gen{\delta_g\ast W,W}\choose 3}.
\end{align*}
By \eqref{eq:36}, the right-hand side of the last line is at most $\norm{\mu}_\infty$ times the number of pairs of distinct triples in $W^3$, thus
\begin{equation}
  \label{eq:38}
  3!  {\gen{\mu\ast W,W}\choose 3}\leq (3!)^2\norm{\mu}_\infty {|W|\choose 3}^2.
\end{equation}
If $\gen{\mu\ast W,W}\geq 4$, then the left-hand side of \eqref{eq:38} is at least $\gen{\mu\ast W,W}^3/4$, so
\[
\gen{\mu\ast W,W}^3\leq 4\norm{\mu}_{\infty}|W|^6,
\]
which proves \eqref{eq:34}.

To prove \eqref{eq:35}, we decompose $f_W$ into its positive and negative parts:
\begin{equation}
  \label{eq:39}
  f_W(x) = (1-\alpha)W(x) - \alpha W^c(x)
\end{equation}
where $\alpha=|W|/(p+1)$ and $W^c=\prob^1(\F_p)\setminus W$ is the complement of $W$.
It follows that
\begin{equation}
  \label{eq:40}
  \gen{\mu\ast f_W,f_W}\leq (1-\alpha)^2 \gen{\mu\ast W,W}+ \alpha^2\gen{\mu\ast W^c,W^c}.
\end{equation}
By the first part of the theorem, we have either
$\gen{\mu\ast W,W}<4$ or
\[
(1-\alpha)^2 \gen{\mu\ast W,W}\leq (1-\alpha)^2 2\norm{\mu}_\infty^{1/3}|W|^2
=2\norm{\mu}_\infty^{1/3}\norm{f_W}_2^2,
\]
and either
$\gen{\mu\ast W^c,W^c}<4$ or
\[
\alpha^2 \gen{\mu\ast W^c,W^c}\leq \alpha^2 2\norm{\mu}_\infty^{1/3}(p+1-|W|)^2
=2\norm{\mu}_\infty^{1/3}\norm{f_W}_2^2.
\]

Thus by the above equations and \eqref{eq:40} we have
\[
\gen{\mu\ast f_W,f_W}\leq\max \left(4\norm{\mu}_\infty^{1/3}\norm{f_W}_2^2, 2\norm{\mu}_\infty^{1/3}\norm{f_W}_2^2+4, 8  \right).
\]
The maximum is only achieved by the middle term when all terms are equal to 8, so we have
\[
\gen{\mu\ast f_W,f_W}\leq\max \left(4\norm{\mu}_\infty^{1/3}\norm{f_W}_2^2, 8  \right),
\]
as claimed.
\end{proof}


\bibliographystyle{abbrvurl}
\bibliography{/Users/brendan/Dropbox/library}

\end{document}